\documentclass[12pt]{amsart}
\usepackage{amsmath,amsfonts,amssymb,amsthm}
\usepackage{anysize}
\marginsize{2cm}{2cm}{2cm}{2cm}
\usepackage{graphicx}

\def\w{\omega}

\def\u{\mathfrak{u}}
\def\vg{\mathfrak{v}}

\def\g{\mathfrak{g}}

\def\h{\mathfrak{h}}

\def\o{\theta}

\def\r{\mathbb{R}}

\def\aff{\mathfrak{aff}}

\def\^{\wedge}

\def\c{\mathbb{C}}

\newcommand{\nc}{\newcommand}
\nc{\ZZ}{{\mathbb Z}}
\nc{\NN}{{\mathbb N}}
\nc{\ad}{\operatorname{ad}}
\nc{\tr}{\operatorname{tr}}
\nc{\I}{\operatorname{Id}}
\nc{\alt}{\raise1pt\hbox{$\bigwedge$}}
\nc{\pint}{\langle \cdotp,\cdotp \rangle }
\nc{\la}{\langle}
\nc{\ra}{\rangle}
\nc{\n}{\noindent}

\theoremstyle{plain}
\newtheorem{teo}{\bf Theorem}[section]
\newtheorem{cor}[teo]{\bf Corollary}
\newtheorem{prop}[teo]{\bf Proposition}
\newtheorem{lema}[teo]{\bf Lemma}

\theoremstyle{definition}

\newtheorem{ejemplo}[teo]{\bf Example}

\theoremstyle{remark}
\newtheorem{rem}[teo]{Remark}
\newtheorem{rems}[teo]{Remarks}

\newcommand{\ri}{{\rm (i)}}
\newcommand{\rii}{{\rm (ii)}}
\newcommand{\riii}{{\rm (iii)}}
\newcommand{\riv}{{\rm (iv)}}
\newcommand{\rv}{{\rm (v)}}

\title[Lattices in LCK and LCS almost abelian Lie groups]{Lattices in almost 
abelian Lie groups with locally conformal K\"ahler or symplectic structures}

\author{A. Andrada}
\email{andrada@famaf.unc.edu.ar}
\author{M. Origlia}
\email{origlia@famaf.unc.edu.ar}

\date{}
\address{FaMAF-CIEM, Universidad Nacional de C\'{o}rdoba, Ciudad Universitaria, 5000 C\'{o}rdoba, 
Argentina}


\begin{document}

\begin{abstract}
We study the existence of lattices in almost abelian Lie groups that admit left invariant locally 
conformal K\"ahler or locally conformal symplectic structures in order to obtain compact 
solvmanifolds equipped with these geometric structures. In the former case, we show that such 
lattices exist only in dimension $4$, while in the latter case we provide examples of such Lie 
groups admitting lattices in any even dimension.
\end{abstract}

\maketitle

\section{Introduction}

\

The aim of this article is to study the existence of lattices in a family of solvable Lie groups 
equipped with certain left invariant geometric structures, namely, locally conformal K\"ahler 
structures and locally conformal symplectic structures. 

We recall that a $2n$-dimensional Hermitian manifold $(M,J,g)$ is called {\it locally conformal 
K\"ahler} (LCK for short) if $g$ can be rescaled locally, in a neighborhood of any point in $M$, so 
as to be K\"ahler. If $\omega$ denotes the fundamental $2$-form of $(J,g)$ defined by 
$\omega(X,Y)=g(JX,Y)$ for any $X,Y$ vector fields on $M$, it is well known that the LCK condition 
is equivalent to the existence of a closed $1$-form $\theta$ on $M$ such that 
$d\omega=\theta\wedge\omega$. This $1$-form $\theta$ is called the {\it Lee form}. These manifolds 
are a natural generalization of K\"ahler manifolds, and they have been much studied by 
many authors since the work of I. Vaisman in the '70s (see for instance \cite{DO, GMO, MO, O, OV, V1}).

A generalization of this class of manifolds is given by the so-called {\it locally conformal
symplectic} (LCS) manifolds, i.e., those manifolds carrying a non-degenerate $2$-form $\omega$
satisfying $d\omega=\theta\wedge\omega$ for some closed $1$-form $\theta$. Locally conformal
symplectic manifolds were considered by Lee in \cite{L} and they have been firstly studied by
Vaisman in \cite{V}. Some recent results can be found in \cite{Ba,Ha,HR,LV}, among others.

\smallskip

Any left invariant LCK or LCS structure on a simply connected Lie group $G$ gives rise naturally 
to an LCK or LCS structure on a quotient $\Gamma\backslash G$ of $G$ by a discrete subgroup. In 
this article we will consider solvable Lie groups $G$ and discrete subgroups $\Gamma$ such that 
$\Gamma\backslash G$ is compact, i.e., $\Gamma$ is a lattice in $G$; the compact quotient is called 
a solvmanifold. In general, it is difficult to determine whether a given Lie group admits lattices. 
However, there are two classes of Lie groups for which a criterion exists: for nilpotent Lie 
groups, there is the well known Malcev's theorem \cite{Ma}, while for almost abelian Lie groups 
there is a characterization by Bock (\cite{B}, see Proposition \ref{latt} below). We recall that a Lie group 
is called {\it almost abelian} when its Lie algebra has an abelian ideal of codimension one. Left 
invariant LCK or LCS structures on nilpotent Lie groups have been thoroughly studied in \cite{
BM,S}, therefore we will focus on the class of almost abelian Lie groups. This class has appeared 
recently in several different contexts (see for instance \cite{AGMP, B, CM, LR,LW}).

\smallskip

LCK and LCS structures on Lie groups and Lie algebras (and also in their compact quotients 
by discrete subgroups, if they exist) have been considered by several authors lately. For instance, 
it was shown in \cite{S} that an LCK nilpotent Lie algebra is isomorphic to 
$\mathfrak h_{2n+1}\times\mathbb R$, where $\mathfrak h_{2n+1}$ denotes the $(2n+1)$-dimensional 
Heisenberg Lie algebra. In \cite{K} it is proved that some solvmanifolds with left invariant 
complex structures do not admit Vaisman metrics, i.e. LCK metrics with parallel Lee form. 
Moreover, it is shown that some Oeljeklaus-Toma manifolds (used to disprove a conjecture by I. 
Vaisman) are in fact solvmanifolds with invariant LCK structures. In \cite{ACHK} it is proved 
that the any reductive Lie group admitting left invariant LCK metrics are locally isomorphic to 
$U(2)$ and $GL(2,\r)$. In \cite{AO} the authors prove that if a solvmanifold with an abelian 
complex structure (i.e. $[Jx,Jy]=[x,y]$ for any $x,y$ in the corresponding Lie algebra) admits an 
invariant LCK metric then the Lie algebra is isomorphic to 
$\mathfrak h_{2n+1}\times\mathbb R$. In \cite{BM} a structure theorem for Lie 
algebras admitting LCS structures of the first kind is given, and $6$-dimensional LCS nilpotent Lie 
algebras are classified. 

\smallskip

The outline of this article is as follows. In Section $2$ we review some known results about LCK
and LCS structures and almost abelian Lie groups. 

In Section $3$ we characterize the Lie algebras of the almost abelian Lie groups that admit a left 
invariant LCK structure (Theorem \ref{lcK}) in any even dimension $\geq 4$. Moreover, we determine 
whether these Lie groups admit lattices, proving that this happens only in dimension $4$ (Theorem 
\ref{6-lattices}), using the criterion for the existence of lattices in almost abelian Lie groups 
given in \cite{B}. In particular, there is a one-parameter family of unimodular almost abelian 
$4$-dimensional Lie groups with left invariant LCK structures, with countably many of them 
admitting lattices (Theorem \ref{4-lattice}). The $4$-dimensional solvmanifolds thus obtained are 
Inoue surfaces of type $S^0$ (see \cite{H}).

In Section $4$ we study LCS structures on almost abelian Lie algebras and we get a characterization 
of these Lie algebras in any even dimension (see Theorem \ref{lcs} for dimension $6$ or higher and 
Theorem \ref{4-lcs} for dimension $4$). Furthermore, we establish which of these LCS structures 
are of the first kind or of the second kind. We also determine whether the $4$-dimensional Lie 
groups associated to these Lie algebras admit lattices (Theorem \ref{4-lcs-latt}) and finally, 
we build in each dimension greater than or equal to $6$ a family of almost abelian solvmanifolds 
admitting an LCS structure of the second kind which carry no invariant complex structure. Moreover, 
we determine the Betti numbers associated to the de Rham cohomology and the adapted (or 
Lichnerowicz) cohomology of these solvmanifolds in low dimensions. 

\

\noindent \textbf{Acknowledgements.} We would like to thank A.~Moroianu for his help with  
integer polynomials and I.~Dotti for useful comments and suggestions.
The authors were partially supported by CONICET, ANPCyT and SECyT-UNC 
(Argentina).

\

\section{Preliminaries}

\subsection{Locally conformal structures on manifolds}

Let $(M,J,g)$ be a $2n$-dimensional Hermitian manifold, where $J$ is a complex structure and
$g$ is a Hermitian metric. $(M,J,g)$ is a {\em locally conformal K\"ahler} manifold (LCK
for short) if there exists an open covering $\{ U_i\}_{i\in I}$ of $M$ and a family $\{ f_i\}_{i\in
I}$ of $C^{\infty}$ functions, $f_i:U_i \to \r$, such that each local metric 
\begin{equation}\label{gi} 
g_i=\exp(-f_i)\,g|_{U_i} 
\end{equation} 
is K\"ahler. Also $(M,J,g)$ is {\em globally conformal K\"ahler} (GCK) if there exists a
$C^{\infty}$ function, $f:M\to\r$, such that the metric $\exp(-f)g$ is K\"ahler.

An equivalent characterization of an LCK manifold can be given in terms of the fundamental form
$\omega$, which is defined by $\omega(X,Y)=g(JX,Y)$, for all $X,Y \in \mathfrak{X}(M)$. Indeed, a
Hermitian manifold $(M,J,g)$ is LCK if and only if there exists a closed $1$-form $\theta$
globally defined on $M$ such that 
\begin{equation}\label{lck}
d\omega=\theta\wedge\omega.
\end{equation} 
This closed $1$-form $\theta$ is called the \textit{Lee form} (see \cite{L}). The Lee form is 
completely determined by $\omega$, and there is an explicit formula given by
\begin{equation}\label{lee}
\theta=-\frac{1}{n-1}(\delta\omega)\circ J, 
\end{equation} 
where $\delta$ is the codifferential and $2n$ is the dimension of $M$.

It can be seen that $(M,J,g)$ is LCK if and only if
\[(\nabla_XJ)Y=\frac{1}{2}\left\{\theta(JY)X-\theta(Y)JX+g(X,Y)J\theta^\#+\omega(X,Y)\theta^\#\right
\}, \] for all 
$X,Y\in \mathfrak{X}(M)$,
where $\theta^\#$ is the dual vector field of the $1$-form $\theta$ and $\nabla$ is the Levi-Civita
connection associated to $g$. This shows that LCK manifolds belong to the class $\mathcal{W}_4$
of the Gray-Hervella classification of almost Hermitian manifolds \cite{GH}.

\

A generalization of an LCK structure is given by a \textit{locally conformal symplectic} 
structure (LCS for short), that is, a non degenerate $2$-form $\omega$ on the manifold $M$ such that 
there exists an open cover $\{U_i\}$ and smooth functions $f_i$ on $U_i$ such that 
\[\omega_i=\exp(-f_i)\omega\] 
is a symplectic form on $U_i$. This condition is equivalent to requiring that
\[d\omega=\theta\wedge\omega\] for some closed $1$-form $\theta$, called again the Lee form. 
Moreover, $M$ is called globally conformal symplectic (GCS) if there exist  a $C^{\infty}$ 
function, $f:M\to\r$, such that $\exp(-f)g$ is a symplectic form. Equivalently, $M$ is a GCS 
manifold if there exists a exact $1$-form $\theta$ globally defined on $M$ such that 
$d\omega=\theta\wedge\omega$.   

We note that in the LCS case the Lee form is also uniquely determined by the non degenerate 
$2$-form $\omega$, but there is not an explicit formula for $\theta$ such as \eqref{lee} in the LCK 
case. The pair $(\omega, \theta)$ will be called an LCS structure on $M$.

It is well known that
\begin{itemize}
 \item If $(\omega,\theta)$ is an LCS structure on $M$, then $\omega$ is symplectic if and only if 
$\theta=0$. Indeed, $\theta\wedge\omega=0$ and $\omega$ non degenerate imply $\theta=0$. 
Accordingly, an LCK structure $(J,g)$ is K\"ahler if and only if $\theta=0$.
 \item If $\omega$ is a non degenerate $2$-form on $M$, with $\dim M\ge 6$, such that \eqref{lck}
holds for some $1$-form $\theta$ then $\theta$ is automatically closed and therefore $M$ is LCS.
\end{itemize}

\medskip

We recall next a definition due to Vaisman (\cite{V}). If $(\omega, \theta)$ is an LCS 
structure on $M$, a vector field $X$ is called an infinitesimal automorphism of 
$(\omega,\theta)$ if $\textrm{L}_X\omega=0$, where $\textrm{L}$ denotes the Lie derivative. This 
implies $\textrm{L}_X\theta=0$ as well and, as a consequence, $\theta(X)$ is a constant function on 
$M$. If there exists an infinitesimal automorphism $X$ such that $\theta(X)\neq 0$, the LCS 
structure $(\omega,\theta)$ is said to be of {\em the first kind}, and it is of {\em the second 
kind} otherwise.

Since $\theta$ is closed, we can deform the de Rham differential $d$ to obtain the adapted 
differential operator 
\[d_\theta \alpha= d\alpha -\theta\wedge\alpha.\]
This operator satisfies $d_\theta^2=0$, thus it defines the {\em adapted cohomology} 
$H_\theta^*(M)$ of $M$ relative to the closed $1$-form $\theta$, also called the Lichnerowicz 
cohomology. It is known that if $M$ is a compact oriented $n$-dimensional manifold, then 
$H_\theta^0(M)= H_\theta^n(M)=0$ for any non exact closed $1$-form $\theta$ (see for instance 
\cite{GL,Ha}). For any LCS structure $(\omega,\theta)$ on $M$, the $2$-form $\omega$ 
defines a cohomology class $[\omega]_\theta\in H_\theta^2(M)$, since 
$d_\theta\omega=d\omega-\theta\wedge \omega=0$. It was proved in \cite{V} that if the LCS structure 
is of the first kind then $\omega$ is $d_\theta$-exact, i.e. $[\omega]_\theta=0$. 

\

\subsection{Left invariant LCK and LCS structures on Lie groups and its compact quotients}

Let $G$ be a Lie group with a left invariant Hermitian structure $(J,g)$. If $(J,g)$ satisfies 
the LCK condition \eqref{lck}, then $(J,g)$ is called a {\em left invariant LCK structure} on the 
Lie group $G$. Clearly, the fundamental $2$-form is left invariant and, using \eqref{lee}, it is 
easy to see that the corresponding Lee form $\theta$ on $G$ is also left invariant. 
 
This fact allows us to define LCK structures on Lie algebras. We recall that a {\em complex
structure J} on a Lie algebra $\g$ is an endomorphism $J: \g \to \g$ satisfying $J^2=-\I$ and 
\[ N_J=0, \quad  \text{where} \quad N_J(x,y)=[Jx,Jy]-[x,y]-J([Jx,y]+[x,Jy]),\]        
for any $x,y \in \g$. 

Let $\g$ be a Lie algebra, $J$ a complex structure and $\pint$ a Hermitian inner product on $\g$, 
with $\w\in\alt^2\g^*$ the fundamental $2$-form. We say that $(\g,J,\pint)$ is {\em locally 
conformal K\"ahler} (LCK) if there exists $\theta \in \g^*$, with $d\theta=0$, such that
\begin{equation} \label{g-lck-0}
d\w=\o\wedge\w.
\end{equation}

In the same fashion, an LCS structure $(\omega,\theta)$ on a Lie group $G$ is called left invariant 
if $\omega$ is left invariant, and this easily implies that $\theta$ is also left invariant. 
Accordingly, we say that a Lie algebra $\g$ admits a {\em locally conformal symplectic} (LCS) 
structure if there exist $\omega\in\alt^2\g^*$ and $\theta \in \g^*$, with $\omega$ non degenerate 
and $\theta$ closed, such that \eqref{g-lck-0} is satisfied.

As in the case of manifolds we have that an LCS structure $(\omega,\theta)$ on a Lie algebra $\g$ 
can be of the first kind or of the second kind. Indeed, let us denote by $\g_\omega$ the set of 
infinitesimal automorphisms of the LCS structure, that is, 
\begin{equation}\label{autom}
\g_\omega = \{x\in\g: \textrm{L}_x\omega=0\} = \{x\in\g: \omega([x,y],z)+\omega(y,[x,z])=0 \; \text{for all} \; y,z\in\g\}.
\end{equation}
Note that $\g_\omega \subset \g$ is a Lie subalgebra, thus the restriction of $\theta$ to 
$\g_\omega$ is a Lie algebra morphism called {\em Lee morphism.} The LCS structure $ 
(\omega,\theta)$ is said to be {\em of the first kind} if the Lee morphism is surjective, and {\em 
of the second kind} if it is identically zero (see \cite{BM}).

\smallskip

For a Lie algebra $\g$ and a closed $1$-form $\theta\in\g^*$ we also have the adapted cohomology 
$H_\theta^*(\g)$ defined by the differential operator \[d_\theta \alpha= d\alpha 
-\theta\wedge\alpha,\] 
on $\alt^* \g^*$. According to \cite{Mil}, this adapted cohomology coincides with the Lie algebra 
cohomology of $\g$ with coefficients in a $1$-dimensional $\g$-module $V_{\theta}$, where the 
action of $\g$ on $V_{\theta}$ is given by
\begin{equation}\label{act}
 Xv=-\theta(X)v, \quad X\in\g,\, v\in V_{\theta}.
\end{equation}
The fact that $\theta$ is closed guarantees that this is a Lie algebra representation. 

\medskip

If the Lie group is simply connected then any left invariant LCK or LCS structure turns 
out to be globally conformal to a K\"ahler or symplectic structure. Therefore we will study 
compact quotients of such a Lie group by discrete subgroups, which will be non 
simply connected and will inherit an LCK or LCS structure. Recall that a discrete subgroup 
$\Gamma$ of a simply connected Lie group $G$ is called a \textit{lattice} if the quotient 
$\Gamma\backslash G$ is compact.  The quotient $\Gamma\backslash G$ is known as a solvmanifold if 
$G$ is solvable and as a nilmanifold if $G$ is nilpotent, and in these cases we have that 
$\pi_1(\Gamma\backslash G)\cong \Gamma$. We note that an LCS structure of the first kind on a Lie 
algebra induces an LCS structure of the first kind on any compact quotient of the corresponding 
simply connected Lie group by a discrete subgroup.

\smallskip

In the case when $G$ is completely solvable, i.e. it is a solvable Lie group such that 
the endomorphisms $\ad_X$ of its Lie algebra $\g$ have only real eigenvalues for all $X\in \g$, the
de Rham and adapted
cohomology of $\Gamma\backslash G$ can be computed in terms of the cohomology of $\g$. Indeed, 
Hattori proved in \cite{Hat} that if $V$ is a finite dimensional triangular\footnote{A $\g$-module 
$V$ is called triangular if the endomorphisms of $V$ defined by $v\mapsto Xv$ have only real 
eigenvalues for any $X\in\g$.} $\g$-module, then $\overline{V}:=C^\infty(\Gamma\backslash 
G)\otimes V$ is a $\mathfrak{X}(\Gamma\backslash G)$-module and there is an isomorphism 
\begin{equation}\label{kill}
 H^*(\g,V) \cong H^*(\mathfrak{X}(\Gamma\backslash G), \overline{V}).
\end{equation}
Therefore:
\begin{itemize}
 \item If $V=\r$ is the trivial $\g$-module, then the right-hand side in \eqref{kill} gives 
the usual de Rham cohomology of $\Gamma\backslash G$, so that 
\begin{equation}\label{deRham}
H^*(\g) \cong H^*_{dR}(\Gamma\backslash G).
\end{equation}
 \item If $V=V_\theta$ with the action given by \eqref{act}, then we can identify 
$\overline{V}$ with $C^\infty(\Gamma\backslash G)$ and the action of $\mathfrak{X}(\Gamma\backslash 
G)$ on $C^\infty(\Gamma\backslash G)$ is given by 
\[ X\cdot f= Xf-\theta(X)f, \qquad X\in\g, \, f\in C^\infty(\Gamma\backslash G).\]
Here we are using that there is a natural inclusion $\g \hookrightarrow 
\mathfrak{X}(\Gamma\backslash G)$ and a bijection $C^\infty(\Gamma\backslash G)\otimes \g \to 
\mathfrak{X}(\Gamma\backslash G)$ given by $f\otimes X \mapsto fX$. As a consequence, in this case 
\eqref{kill} becomes (cf. \cite[Corollary 4.1]{Mil})
\begin{equation}\label{kill_adapted}
 H^*_\theta(\g)\cong H^*_\theta(\Gamma\backslash G).
\end{equation}
\end{itemize}
In particular, $H^*_{dR}(\Gamma\backslash G)$ and $H^*_\theta(\Gamma\backslash G)$ do not depend on 
the lattice $\Gamma$.

\

\subsection{Almost abelian Lie groups}

In this article we will focus on a family of solvable Lie groups, namely, the almost abelian ones. 
We recall that a Lie group $G$ is said to be {\em almost abelian} if its Lie algebra $\g$ has a 
codimension one abelian ideal. Such a Lie algebra will be called almost abelian, and it can be 
written as $\g= \r f_1 \ltimes_{\ad_{f_1}} \mathfrak{u}$, where $\mathfrak u$ is an abelian ideal of 
$\g$, and $\r$ is generated by $f_1$. Accordingly, the Lie group $G$ is a semidirect product 
$G=\r\ltimes_\phi \r^d$ for some $d\in\mathbb N$, where the action is given by 
$\phi(t)=e^{t\ad_{f_1}}$. We point out that an almost abelian Lie algebra is nilpotent if and only 
if the operator $\ad_{f_1}|_{\mathfrak u}$ is nilpotent. 

Regarding the isomorphism classes of almost abelian Lie algebras, it can be proved that  
\begin{lema}\label{ad-conjugated}
Two almost abelian Lie algebras $\g_1=\r f_1\ltimes_{\ad_{f_1}} \u_1$ and 
$\g_2=\r f_2\ltimes_{\ad_{f_2}}\u_2$ are isomorphic if and only if there exists $c\neq 0$ such that
$\ad_{f_1}$ and $c\ad_{f_2}$ are conjugate. 
\end{lema}
See \cite{ABDO} for a proof in the case $\dim\g=4$. 

\medskip

\begin{rem}
 Note that a codimension one abelian ideal of an almost abelian Lie algebra is almost always 
unique. Indeed, if $\u$ and $\vg$ are codimension one abelian ideals of $\g$, with $\u\neq 
\vg$, then $\u\cap\vg$ is a codimension two abelian ideal, and we can decompose $\g=\r u_0\oplus \r 
v_0 \oplus \u\cap \vg$ for some $u_0\in\u-\vg$, $v_0\in\vg-\u$. If $[u_0,v_0]=0$, then $\g$ is 
abelian, and if $[u_0,v_0]\neq 0$, then $[u_0,v_0]$ generates the commutator ideal of $\g$, and 
therefore $\g$ is isomorphic to $\h_3\times \r^s$ for some $s\geq 0$, where $\h_3$ denotes the 
$3$-dimensional Heisenberg Lie algebra.

As a consequence, in any other case the codimension one abelian ideal is unique.
\end{rem}

\medskip

An important feature concerning almost abelian Lie groups is that there exists a 
criterion to determine when such a Lie group admits lattices. In general, it is not easy to 
determine if a given Lie group $G$ admits a lattice. A well known restriction 
is that if this is the case then $G$ must be unimodular (\cite{Mi}), i.e. the Haar measure on $G$ 
is left and right invariant, or equivalently, when $G$ is connected, $\tr(\ad_x)=0$ for any $x$ in 
the Lie algebra $\g$ of $G$. In the case of an almost abelian Lie group we have the following fact, 
which will prove very useful in forthcoming sections:

\begin{prop}\label{latt}\cite{B}
Let $G=\r\ltimes_\phi\r^{2n+1}$ be an almost abelian Lie group. Then $G$ admits a lattice if and 
only if there exists a
$t_0\neq 0$ such that $\phi(t_0)$ can be conjugated to an integer matrix.
\end{prop}
In this situation, a lattice is given by $\Gamma=t_0\mathbb{Z}\ltimes P^{-1}\mathbb Z^{2n+1}$, 
where $P\phi(t_0)P^{-1}$ is an integer matrix.

\

\section{Lattices in almost abelian Lie groups with LCK structures}

In this section we characterize firstly all almost abelian Lie algebras which admit LCK
structures. Secondly, we determine which of the associated simply connected almost abelian Lie 
groups admit lattices, proving that this only happens in dimension $4$.

\subsection{LCK almost abelian Lie algebras}

Let $\g$ be an almost abelian Lie algebra of dimension $2n+2$, so that there exists an abelian
ideal $\u$ of dimension $2n+1$. Assume that $\g$ is equipped with a Hermitian structure $(J,\pint)$,
where $J$ is a complex structure. Consider $\u\cap J\u$, the maximal $J$-invariant subspace of $\u$.
Clearly, $\dim(\u\cap J\u)=2n$, and there exists $f_2\in\u$, $f_2\in(\u\cap J\u)^\perp$, and
$|f_2|=1$. Define $f_1=-Jf_2\in \u^\perp$.
Then we get the orthogonal decomposition $\g=\text{span}\{f_1,f_2\}\oplus (\u\cap J\u)$, where 
$\u=\r f_2\oplus (\u\cap J\u)$. 

We can write also $\g= \r f_1\ltimes \u$, where the adjoint action of $f_1$ on $\u$ is given by 
\begin{equation}\label{w}
[f_1,f_2]=\mu f_2+ v_0, 
\end{equation}
for some $v_0\in\u\cap J\u$ and $\mu\in\mathbb R$, and for $x\in \u\cap J\u$ we have 
\[ [f_1,x]=\eta(x) f_2+ Ax, \quad \eta\in(\u\cap J\u)^*, \, A\in\operatorname{End}(\u\cap J\u).\]

Since $J$ is integrable, we have that $N_J(f_1,x)=0$ for all $x \in \u\cap J\u$, that is, \[ 
J[f_1,x]=[Jf_1,x]+[f_1,Jx]+J[Jf_1,Jx],\] which implies
\[ -\eta(x)f_1+JAx=\eta(Jx)f_2+AJx, \]
hence we obtain $\eta=0$ and $JA=AJ$. Therefore, $\u\cap J\u$ is a $J$-invariant abelian ideal of 
codimension $2$ in $\g$. Denoting $\mathfrak a: = \u\cap J\u$, we obtain the following result 
(cf. \cite{LR}).

\

\begin{lema}\label{hermitian}
Let $\g$ be an almost abelian Lie algebra and $(J, \pint)$ a Hermitian structure on $\g$. Then
there exist a $J$-invariant abelian ideal $\mathfrak a$ of codimension $2$, an orthonormal basis
$\{f_1,f_2\}$ of ${\mathfrak a}^{\perp}$, $v_0\in\mathfrak a$ and $\mu\in\mathbb R$ such that
$[f_1,f_2]=\mu f_2 + v_0$, $\ad_{f_1}|_{\mathfrak a}$ commutes with $J|_{\mathfrak a}$ and 
$\ad_{f_2}|_{\mathfrak a}=0$.
\end{lema}

\


We will assume from now on that $(J, \pint)$ is LCK. Therefore there exists a closed $1$-form 
$\theta\neq 0$ such that $d\omega=\theta\wedge \omega$. We will consider two
cases, according to the dimension of $\g$.

\medskip

\subsubsection{Dimension of $\g\geq 6$}
In this case, that is, $n\geq 2$, we have that $\dim(\u\cap J\u)\geq 4$.

\

Let us prove that $\theta$ is a multiple of the closed $1$-form $f^1$, where $f^1$ is the metric 
dual of $f_1$. Indeed, for each $x\in \u\cap J\u$ we can find $0\neq y\in \u\cap J\u$ such that 
$\la x,y\ra=0=\la x,Jy \ra$. Since $\u$ is abelian, we have that $d\omega(x,y,Jy)=0$, while on the 
other hand we compute $\theta\wedge\omega(x,y,Jy)=\theta(x)\omega(y,Jy)=\theta(x)|y|^2$, and as a 
consequence we obtain $\theta(x)=0$ for any $x\in \u\cap J\u$. 

Now, from $d\omega(f_2,x,Jx)=\theta\wedge \omega(f_2,x,Jx)$ for $x\in\u\cap J\u, x\neq0$, we obtain
that $\theta(f_2)=0$. Consequently, $\theta=a f^1$ for some $a\neq 0$.

\

From $d\omega(f_1,f_2,x)=\theta\wedge \omega(f_1,f_2,x)$ and \eqref{w} we obtain that
$\theta(x)=-\la Jv_0,x\ra$ for any $x\in \u\cap J\u$. This implies that $Jv_0=0$ and therefore 
\[v_0=0, \quad \text{so that} \quad [f_1,f_2]= \mu f_2.\] 
This implies that $\g={\mathfrak a}^\perp\ltimes \mathfrak a$, where we are using the notation of
Lemma \ref{hermitian}.

Let us compute $d\omega(f_1,x, Jy)=\theta\wedge \omega(f_1,x, Jy)$ for $x,y\in \u\cap J\u$. We 
obtain that 
\[ \langle Ax,y\rangle +\langle x,Ay\rangle = -a\langle x,y\rangle. \]
Decomposing $A$ as $A=U+B$, where $U$ is self-adjoint and $B$ is skew-adjoint, it follows from 
the equation above that $U=-\frac{a}{2}\operatorname{Id}$, and therefore, setting 
$\lambda=-\frac{a}{2}$, we have that
\[ A=\lambda \operatorname{Id}+B, \quad B^*=-B, \quad BJ=JB. \]

Choosing an orthonormal basis $\{u_1,\dots,u_n,v_1,\dots,v_n\}$ of $\u\cap J\u$ such that 
$Ju_i=v_i$, $i=1,\ldots,n$, we can identify $\u\cap J\u$ with $\r^{2n}$, $\u$ with 
$\r^{2n+1}$, $\g$ with $\r\ltimes\r^{2n+1}$, and we have the following matrix representations: 
\begin{equation}\label{ad-f1} 
J|_{\r^{2n}}=\begin{pmatrix} 0 & -I \cr I & 0\end{pmatrix}, \qquad 
\ad_{f_1}|_{\r^{2n+1}}=
\left(
\begin{array}{c|ccc}       
\mu \\
\hline
  & & & \\
  &  & \lambda I + B & \\
  & & & 
\end{array}
\right), \qquad B\in \mathfrak{u}(n).
\end{equation}
Moreover, the fundamental $2$-form $\omega$ and the Lee form $\theta$ are given by: 
\[\w=f^1\wedge f^2 + \sum_{i=1}^n u^i\wedge v^i, \qquad \theta = -2\lambda f^1, \] 
where $\{f^1, f^2, u^1,\dots,u^n,v^1,\dots,v^n\}$ is the dual basis. Note that the operator 
$\ad_{f_1}|_{\r^{2n+1}}$ is nilpotent if and only if it is zero, and that 
if $\g$ is unimodular then $\lambda=-\frac{1}{2n} \mu$.

\medskip

\begin{rems}
(i) When $\lambda=0$, it follows that $\theta=0$ and therefore the Hermitian structure $(J, \pint)$ 
is K\"ahler. Since we are interested in non-K\"ahler LCK structures, we will assume $\lambda\neq 
0$, and in this case, it is easy to see that the operator $\lambda\operatorname{Id}+B$ with 
$B\in\mathfrak{u}(n)$ is non-singular.

\smallbreak\n(ii) Almost K\"ahler structures on almost abelian Lie algebras were 
studied in \cite{LW}.
\end{rems}

\

\subsubsection{Dimension of $\g=4$.}

In this case we have $\g=\text{span}\{f_1,f_2\}\oplus \text{span}\{u,v\}$, where $Jf_1=f_2$,
$Ju=v$ and $\{f_1,f_2,u,v\}$ is an
orthonormal basis of $\g$. The brackets on $\g$ are given by
\[ [f_1,f_2]=\mu f_2 + mu +nv, \; \text{for some }  \mu,m,n\in\mathbb{R},\]
\[ [f_1,z]=Az, \; \text{where} \; z\in\text{span}\{u,v\} \;\text{and}\; A=\begin{pmatrix} x & -y \cr
y & x\end{pmatrix}, \,
\text{with } x,y\in\mathbb{R}.\]
Then we get that
\[ \ad_{f_1}|_{\r^3}=
\left(
\begin{array}{c|cc}       
\mu \\
\hline
  m& x& -y\\
  n& y&  x
\end{array}
\right).
\]

\smallskip

From $d\omega(f_1,f_2,u)=\theta\wedge \omega(f_1,f_2,u)$, we obtain that $\theta(u)=n$; in the same 
way, $\theta(v)=-m$. On the other hand, from $d\omega(f_2,z,Jz)=\theta\wedge \omega(f_2,z,Jz)$ for 
$z\in\text{span}\{u,v\}$, we get that
$\theta(f_2)=0$. Therefore we can write $\theta$ as
\begin{equation}\label{tita4}
 \theta= a f^1+nu^*-mv^*,
\end{equation}
for some $a\in\mathbb{R},$ where $\{f^1,f^2,u^*,v^*\}$ is the dual basis of $\{f_1,f_2,u,v\}$. 
Recalling that $\omega=f^1\wedge f^2+u^*\wedge v^*$, it follows from $d\omega=\theta\wedge\omega$ 
that $a=-2x$. 

Next, since $d\theta=0$, we obtain
\[0=d\theta=-2x\, df^1+n\, du^*-m\, dv^*=(-nx+my)f^1\wedge u^* + (ny+mx)f^1\wedge v^*,\]
and we consider two cases according whether to $m^2+n^2\neq 0$ or $m^2+n^2=0$.

In the first case we have $x=y=0$, thus the only non-vanishing bracket is
$[f_1,f_2]=\mu f_2+mu+nv$. If $\mu=0$, this Lie algebra is isomorphic to 
$\mathfrak{h}_3\times\r$, where $\mathfrak{h}_3$ is the $3$-dimensional Heisenberg Lie 
algebra. It is well known that $\mathfrak{h}_3\times\r$ admits LCK structures (\cite{CFL}, 
see also \cite{AO,S}). If $\mu\neq0$, this Lie algebra is isomorphic to $\aff(\r)\times\r^2$, 
where $\aff(\r)$ denotes the non-abelian $2$-dimensional Lie algebra of the group of affine motions 
of the real line. We point out that this decomposition is neither orthogonal nor $J$-invariant. Note 
that these Lie algebras have $1$-dimensional commutator ideal, $\mathfrak{h}_3\times \r$ is 
nilpotent and $\aff(\r)\times\r^2$ is not unimodular.

The other case is $m=n=0$, so that
\begin{equation}\label{ad-dim4} 
\ad_{f_1}|_{\r^3}=\left(\begin{array}{c|cc}
\mu \\
\hline
  & x & -y\\
  & y & x
\end{array}\right), 
\end{equation}
and we obtain from \eqref{tita4} that $\theta= -2x f^1$. Note that $\ad_{f_1}|_{\r^{2}}= x I + B$, 
for some $x\in\r$ and $B\in \mathfrak{u}(1)$.

\

Conversely, it is easy to verify that if $\g$ is an almost abelian Lie algebra where the adjoint 
action of $\r$ on the codimension one abelian ideal is given by either \eqref{ad-f1} or 
\eqref{ad-dim4}, then $\g$ admits an LCK structure. 

\medskip

To close this section, we state the following result which summarizes the results obtained so far:

\begin{teo}\label{lcK}
Let $\g$ be a $(2n+2)$-dimensional almost abelian Lie algebra and $(J,\pint)$ a Hermitian
structure on $\g$, and let $\g'$ denote the commutator ideal $[\g,\g]$ of $\g$.
\begin{enumerate}
 \item [\ri] If $\dim \g'=1$, then $(J,\pint)$ is LCK if and only if $\g$ is isomorphic to 
$\mathfrak{h}_3 \times\r$ or $\aff(\r)\times \r^2$ as above.
 \item [\rii] If $\dim\g'\geq 2$, then $(J,\pint)$ is LCK if and only if $\g$ can 
be decomposed as $\g={\mathfrak a}^\perp\ltimes \mathfrak a$, a $J$-invariant orthogonal sum with a
codimension $2$ abelian ideal $\mathfrak a$, and there exists an orthonormal basis $\{f_1,f_2\}$ of 
${\mathfrak a}^\perp$ such that 
\[ [f_1,f_2]=\mu f_2, \quad f_2=Jf_1, \quad \ad_{f_2}|_{\mathfrak a}=0 \text{ and } 
\ad_{f_1}|_{\mathfrak a}= \lambda I + B, \]
for some $\mu,\lambda\in\r$, $\lambda\neq 0$, and $B\in \mathfrak{u}(n)$. The corresponding Lee 
form is given by $\theta=-2\lambda f^1$. Furthermore, the Lie algebra $\g$ is unimodular if and only 
if $\lambda=-\frac{\mu}{2n}$.
\end{enumerate}
\end{teo}
 
\medskip

\begin{rems}
(i) If we allow $\lambda=0$ in Theorem \ref{lcK}(ii), the Hermitian structures thus obtained are 
K\"ahler.

\smallbreak\n(ii) A Hermitian manifold is called \textit{Vaisman} if it is LCK with parallel Lee
form $\theta$. The left invariant LCK structures obtained on the Lie groups corresponding to the
Lie algebras in Theorem \ref{lcK}(ii) or $\aff(\r)\times \r^2$ from Theorem \ref{lcK}(i) are never 
Vaisman. This can be seen either by a direct computation or by the fact that the endomorphisms 
$\ad_{f_1}$ are not skew-symmetric (see \cite{AO}). On the other hand, any LCK structure on 
$\mathfrak{h}_3\times \r$ from Theorem \ref{lcK}(i) is Vaisman (see \cite{AO,S}).

\smallbreak\n(iii) With the notation used in \cite{ABDO}, the $4$-dimensional Lie algebras 
admitting LCK structures obtained from Theorem \ref{lcK}(ii) correspond to $\mathfrak r_{3,1}\times 
\r$, $\mathfrak r'_{3,\lambda}\times \r$, $\mathfrak r_{4,\mu,\mu}$ and $\mathfrak 
r'_{4,\mu,\lambda}$ for some $\mu\neq 0$ and $\lambda\neq 0$. The Lie algebra $\aff(\r)\times \r^2$ 
from Theorem \ref{lcK}(i) is denoted by $\mathfrak{r}_{3,0}\times\r$ in \cite{ABDO}.
\end{rems}

\

\subsection{Lattices in the associated LCK Lie groups}

In this section, we will consider solvmanifolds associated to the almost abelian Lie algebras 
obtained in the previous section. Therefore we will study the existence of lattices in the simply 
connected Lie groups associated to these Lie algebras. 

Recall that for an almost abelian solvable Lie group $G=\r\ltimes_\phi\r^{2n+1}$ ($n\geq 1$), its 
Lie algebra is $\g=\r\ltimes_{\ad_{f_1}}\r^{2n+1}$ where $\r$ is generated by $f_1$, and the action 
is given by $\phi(t)=e^{t\ad_{f_1}}$.

\

Let $\g$ be a unimodular almost abelian Lie algebra equipped with an LCK structure. If $\g$ is 
nilpotent, then it follows from Theorem \ref{lcK} that $\g$ is isomorphic to $\mathfrak{h}_3\times 
\r$. The lattices in the associated simply connected nilpotent Lie group $H_3\times \r$ are 
well known (see \cite{GW} for their classification). The corresponding compact nilmanifolds are 
primary Kodaira surfaces \cite{H}.

From now on, we will consider non-nilpotent unimodular almost abelian Lie algebras with LCK 
structures. According to Theorem \ref{lcK}, such a Lie algebra $\g$ can be decomposed as $\g=\r 
f_1\ltimes \r^{2n+1}$, an orthogonal sum, where $\r^{2n+1}=\r f_2\oplus \r^{2n}$ and 
\[ \ad_{f_1}|_{\r^{2n+1}}=
\left(
\begin{array}{c|ccc}       
\mu \\
\hline
 & & &\\
 & & -\frac{\mu}{2n} I + B & \\
 & & &
\end{array}
\right),
\]
for some $0\neq \mu\in\r$ and $B\in \mathfrak{u}(n)$. Note that since $B$ is skew-symmetric, if
$c\in\mathbb C$ is an eigenvalue of $A=-\frac{\mu}{2n}I +B$, then $c=-\frac{\mu}{2n}\pm
i\eta$ for some $\eta\in\mathbb{R}$. 

If $G$ denotes the simply connected almost abelian Lie group with Lie algebra $\g$, then
$G=\r\ltimes_\phi\r^{2n+1}$ with 
\begin{equation}\label{fi}
 \phi(t)=e^{t\ad_{f_1}|_{\r^{2n+1}}}=
\left(
\begin{array}{c|c}       
e^{t\mu} &\\
\hline
\\
& e^{-\frac{t\mu}{2n}}e^{tB}
\\
\\
\end{array}
\right).
\end{equation}
This matrix has a real eigenvalue $e^{t\mu}$ and the others are $e^{-\frac{t\mu}{2n}\pm i\eta}$ for
some $\eta\in\mathbb{R}$.

\medskip

The existence of lattices on $G$ will depend on the dimension of $G$, since we will show that such a 
lattice exists only if $\dim G=4$.

\

\subsubsection{Lattices in dimension $\geq 6$}

We will show that these Lie groups cannot admit lattices for $n\geq 2$. We state first a result about 
the roots of a certain class of polynomials with integer coefficients. 
  
\begin{lema}\label{raices}
Let $p$ be a polynomial of the form \[p(x)=x^{2n+1}-m_{2n}x^{2n}+m_{2n-1}x^{2n-1}+\cdots+m_1x-1\]
with $m_j\in\ZZ$ and $n\geq 2$, and let $x_0,\dots,x_{2n}\in \c$ denote the roots of $p$. If 
$x_0\in\mathbb{R}$ is a simple root and $|x_1|=\dots=|x_{2n}|$, then $x_0=1$ and $|x_j|=1$, 
$j=1,\ldots, 2n$.
\end{lema}

\begin{proof}
Let $\rho\in\mathbb{R}$, $\rho>0$, such that $|x_j|=\rho^{-1}$ for $j=1,\dots,2n$. It follows from 
$\prod_{j=0}^{2n}{x_j}=1$ that $|x_0|=\rho^{2n}$.

Note that we may assume $\rho\geq 1$, since otherwise we consider the reciprocal polynomial 
$p^*(x):=-x^{2n+1}p(x^{-1})$. 

Let us suppose that $\rho >1$.

We will prove first that $p$ is irreducible over $\ZZ[x]$. Indeed, if $p=qr$ with $q,r\in\ZZ[x]$, 
then $x_0$ will be a (simple) real root of one of these polynomials, say $q$, and therefore all the 
roots of $r$ have modulus $\rho^{-1}<1$. The coefficient $r(0)$ will be the product of these roots, 
hence $|r(0)|<1$, and this is a contradiction since $r(0)\in\ZZ-\{0\}$. 

\medskip

Now, let us write $p(x)=\prod_{j=0}^{2n}(x-x_j)$. Expanding this product we obtain:
\[ m_{2n}=x_0+\sum_{j=1}^{2n} x_j, \quad m_1=\frac{1}{x_0}+\sum_{j=1}^{2n}\frac{1}{x_j}. \]
Since $\sum_{j=1}^{2n}\frac{1}{x_j}=\rho^2\sum_{j=1}^{2n} x_j$, we obtain that
\[ m_{2n}-x_0=\frac{1}{\rho^2}\left(m_1-\frac{1}{x_0}\right), \]
which implies, recalling that $|x_0|=\rho^{2n}$,  
\begin{equation}\label{poli-x0}
 \rho^{4n+2}-m_{2n}x_0\rho^{2}+m_1x_0-1=0. 
\end{equation}
 
We consider two different cases, according to: (i) $x_0=\rho^{2n}$, or (ii) $x_0=-\rho^{2n}$.

\medskip

(i) If $x_0=\rho^{2n}$ then \eqref{poli-x0} becomes 
\[ \rho^{4n+2}-m_{2n}\rho^{2n+2}+m_1\rho^{2n}-1=0. \] 
Thus, $y_0:=\rho^2$ is a root of $q(x)=x^{2n+1}-m_{2n}x^{n+1}+m_1x^n-1$.
Let $y_1,\ldots,y_{2n}$ denote the other roots of $q$ and let us consider the polynomial
$\tilde{q}(x)=\prod_{j=0}^{2n}(x-y_j^n)$. Since $q(x)=\prod_{j=0}^{2n}(x-y_j)\in\ZZ[x]$, it can be 
seen that $\tilde{q}\in\ZZ[x]$ as well.

Note that $\tilde{q}(\rho^{2n})=0$. Since $p,\tilde{q}\in\ZZ[x]$ are monic polynomials of the same
degree, both vanish in $x_0=\rho^{2n}$ and $p$ is irreducible, then $p=\tilde{q}$. 

Therefore the set $\{y_1^n,\ldots,y_{2n}^n\}$ is a permutation of $\{x_1,\ldots,x_{2n}\}$, so that
the polynomial $q$ has a simple real root $y_0=\rho^{2}>0$ and the other roots satisfy 
$|y_1|=\dots=|y_{2n}|=\rho^{-\frac1n}$. We can perform the same computations as above with the 
polynomial $q$, taking into account that in this case $\sum_{j=0}^{2n} y_j=0$ and $\sum_{j=0}^{2n} 
\frac{1}{y_j}=0$, since the coefficients of $x^{2n}$ and $x^1$ in $q$ are $0$ (here we are using 
that $n\geq 2$). 

It follows that $\rho$ satisfies the equation 
\[  (\rho^\frac{1}{n})^{4n+2}-1=0, \]
so that $\rho=1$, which contradicts the assumption $\rho>1$. 

\medskip

(ii) If $x_0=-\rho^{2n}$, then \eqref{poli-x0} becomes
\[ \rho^{4n+2}+m_{2n}\rho^{2n+2}-m_1\rho^{2n}-1=0. \] 
Thus, $y_0:=\rho^2$ is a root of $q(x)=x^{2n+1}+m_{2n}x^{n+1}-m_1x^n-1$. Let $y_1,\ldots,y_{2n}$ 
denote the other roots of $q$ and let us consider the polynomial 
$\tilde{q}(x)=\prod_{j=0}^{2n}(x-y_j^n)$. Since $q(x)=\prod_{j=0}^{2n}(x-y_j)\in\ZZ[x]$, then
$\tilde{q}\in\ZZ[x]$ as well. Let us consider now the polynomial $q_1\in \ZZ[x]$ defined by 
$q_1(x)=-\tilde{q}(-x)$. Note that $q_1$ is a monic polynomial which vanishes in $x_0=-\rho^{2n}$. 
It follows from the irreducibility of $p$ that $p=q_1$, but $p(0)=-1$ whereas $q_1(0)=1$, a 
contradiction.

\medskip

We conclude that the assumption $\rho>1$ leads to a contradiction, and as a consequence, we 
have $\rho= 1$. Thus, $|x_j|=1$ for all $j=0,\ldots, 2n$ and $x_0=1$ or $x_0=-1$. 

If $x_0=-1$, then there exists $l\in\{0,1,\ldots,2n\}$ such that $x_1=\cdots=x_{2l}=1$ and the 
remaining 
roots are non-real complex numbers $\alpha_1,\ldots,\alpha_{n-l}$ and their complex conjugates, 
with $|\alpha_j|=1$ for all $j$. However,
\[ 1=\prod_{j=0}^{2n}{x_j}= 
(-1)\left(\prod_{j=1}^{2l}{x_j}\right)\left(\prod_{j=1}^{n-l}|\alpha_j|^2\right)=-1,\]
a contradiction.

Therefore $x_0=1$ and $|x_j|=1$ for all $j$, and the theorem is proved.
\end{proof}

\smallskip

\begin{rem}
 It follows from Kronecker's theorem \cite{Kr} that all the roots of the polynomial $p$ in Lemma 
\ref{raices} are roots of unity.
\end{rem}

\medskip

\begin{teo}\label{6-lattices}
If $G$ is as above with $\mu\neq0$ and $\dim G \geq 6$, i.e. $n\geq 2$, then $G$ admits no lattice.
\end{teo}

\begin{proof}
Suppose that $G$ admits lattices, from Proposition \ref{latt} there exists $t\neq 0$ such that 
$e^{t\ad_{f_1}}$ is conjugate to an
integer matrix. Hence its characteristic polynomial $p$ has integer coefficients and it can be
written as
\[p(x)=x^{2n+1}-m_{2n}x^{2n}+m_{2n-1}x^{2n-1}+\cdots+m_1x-1\] with $m_j\in\ZZ$.
It follows from \eqref{fi} that $p$ has a simple real root $x_0=e^{t\mu}$, and the other roots are 
complex with modulus
$e^{-\frac{t\mu}{2n}}$. It follows from Lemma \ref{raices} that $e^{t\mu}=1$. Since $\mu\neq0$, then 
$t=0$, which is a
contradiction.
\end{proof}

\bigskip

\subsubsection{Lattices in dimension $4$} \label{lat-dim-4}
From Theorem $\ref{lcK}$ we have that $\g=\mathbb R\ltimes\mathbb R^3$ and 
\begin{equation}\label{mu-y}
\ad_{f_1}|_{\r^3}=\left(\begin{array}{c|cc}
\mu \\
\hline
  & -\frac{\mu}{2}& -y\\
  & y&  -\frac{\mu}{2}
\end{array}\right).
\end{equation}
We denote $\g_{(\mu,y)}=(\g,J,\pint)$ where $\ad_{f_1}$ is given by $\eqref{mu-y}$. In the 
non-K\"ahler case, i.e. $\mu\neq 0$, we
get that $\g_{(\mu,y)}$ is isomorphic to $\g_{(1,\frac{y}{\mu})}$. We denote this Lie algebra by 
$\g_b$, where $b=\frac{y}{\mu}$.

\begin{rem}
Note that $\g_b$ and $\g_{-b}$ are isomorphic. Moreover, for $b\neq 0$, $\g_b$ is isomorphic to 
$\mathfrak{r}'_{4,1/b,-{1}/{2b}}$ and $\g_0$ is isomorphic to $\mathfrak{r}_{4,-\frac12,-\frac{1}{2}}$
from \cite{ABDO}, and it follows that they are not pairwise isomorphic for $b\geq 0$.
\end{rem}

\medskip

Let us assume that the simply connected Lie group $G_b$ associated to $\g_b$ admits lattices. Then
according to Proposition \ref{latt} there exists
$t_0\in\mathbb{R}$, $t_0\neq0$, such that $e^{t_0\ad_{f_1}}$ is conjugated to a matrix with integer 
coefficients. Therefore the
characteristic polynomial of $e^{t_0\ad_{f_1}}$ is
\[f(x)=x^3-mx^2+nx-1,\] with $m,n\in\mathbb{Z}$. Note that $f$ has a simple real root $e^{t_0}\neq
1$ and two complex conjugate roots $e^{t_0(-\frac{1}{2}\pm ib)}\in\mathbb C-\mathbb R$. Indeed, if
$e^{t_0(-\frac{1}{2}\pm ib)}\in\mathbb R$, then $e^{-\frac{t_0}{2}}$ (or its opposite) is a double 
real root, and it is easy to see that this implies $e^{-\frac{t_0}{2}}=1$ and therefore $t_0=0$, a 
contradiction. In particular, $G_0$ does not admit lattices.

\medskip

Conversely, we consider $f(x)=x^3-mx^2+nx-1$ with $m,n\in\mathbb{Z}$ such that $f$ has a simple real
root $c\neq 1$ and two complex conjugate roots $\alpha, \overline{\alpha} \in\mathbb{C}-\r$. Then
$|\alpha|^2c=1$, so that $c>0$. If $\alpha=|\alpha|e^{i\phi}$ with $\phi\in(0,\pi)$ we consider the
Lie algebra $\g=\mathbb{R}\ltimes\mathbb{R}^3$ where the action is given by
\[ \ad_{f_1}|_{\r^3}=\left(\begin{array}{c|cc}
-2\log |\alpha|\\
\hline
  & \log |\alpha|& -\phi\\
  & \phi&  \log |\alpha|
\end{array}\right).\]
Then 
\[ e^{\ad_{f_1}|_{\r^3}}=\left(\begin{array}{c|cc}
|\alpha|^{-2}\\
\hline
  & |\alpha|\cos\phi& -|\alpha|\sin\phi\\
  & |\alpha|\sin\phi&  |\alpha|\cos\phi
\end{array}\right).\] 

\

Since this matrix has eigenvalues $c$, $\alpha$ and $\overline\alpha$ and they are all different, we 
have that it is conjugated to
the companion matrix
\[ \begin{pmatrix}
0 & 0 &  1\\
1 & 0 & -n\\
0 & 1 &  m
\end{pmatrix} \in \textit{SL}(3,\mathbb{Z}).\]
According to Proposition \ref{latt} the simply connected Lie group associated to $\g$ admits 
lattices.

Note that as $c\neq 1$, this Lie algebra coincides with $\g_{(-2\log|\alpha|, \phi)}$ as above and
therefore is isomorphic to
$\g_b$ with $b=\frac{\phi}{\log c}$.

\

Let $\Sigma=\{(m,n)\in\mathbb{Z}\times\mathbb{Z}: f_{m,n}(x)= x^3 - mx^2 + nx-1 \;\text{has  
roots} \; c\in\mathbb R, \alpha,\overline\alpha\in\mathbb C-\mathbb R\}$.  This region
$\Sigma$ is the set of the pairs $(m,n)\in\mathbb{Z}\times\mathbb{Z}$ such that the discriminant 
$\Delta_{m,n}$ of $f_{m,n}$ is negative, that is, $\Delta_{m,n}=-27-4m^3+18mn+m^2 n^2-4 n^3<0$ (see 
Figure \ref{fig:ejemplo}). 
 
\begin{figure}
\centering
 \includegraphics[width=0.4\textwidth]{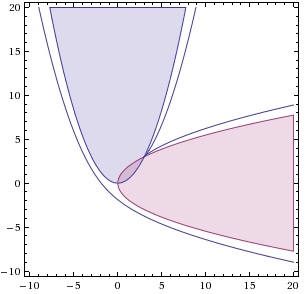}
\caption{Discriminant of $f_{m,n}< 0$}
\label{fig:ejemplo}
\end{figure}

Note that $c=1$ if and only if $m=n$, and in this case the other roots are complex conjugate if and 
only if $m=0,1,2$.

Let $\Sigma'=\Sigma-\{(0,0),(1,1),(2,2)\}$ and for any $k\in \mathbb Z$ consider the function $h_k: 
\Sigma' \to \mathbb{R}$ that assigns to $(m,n)$ the real number $\frac{\phi_k}{\log c}$ where 
$\phi_k=\phi+2k\pi$, $\phi\in(0,\pi)$. With this notation, we may state the following result.

\smallskip

\begin{teo}\label{4-lattice}
The simply connected almost abelian Lie group $G_b$ with Lie algebra $\g_b$ admits lattices if and 
only if $b\in \bigcup_{k\in\mathbb Z}\operatorname{Im} (h_k)$, a countable subset of $\r$.
\end{teo}

\smallskip

\begin{rems}
(i) The solvmanifolds associated to $\g_b$ are Inoue surfaces of type $S^0$ \cite{Kam, T} (see 
\cite{S1} for an explicit construction of a lattice on the Lie groups $G_b$).

\smallbreak\n(ii) Note that if $(m,n)\in\Sigma$ then $(n,m)\in\Sigma$ as well, since 
$f_{n,m}(x)=-x^3 f_{m,n}(\frac1x)$. Thus
$\Sigma$ is symmetric with respect to the diagonal $y=x$.

\smallbreak\n(iii) The region $\Sigma$ contains all pairs $(m,n)$ such that $m^2<3n$ or $n^2<3m$ 
(see Figure \ref{fig:ejemplo}).

\smallbreak\n(iv) It is easy to see that for $(m,n)\in \Sigma$, the real root $c$ of $f_{m,n}$ 
satisfies $c>1$ when $n<m$, therefore $c$ is a Pisot number. Recall that a Pisot number is a real 
algebraic number greater than $1$ such that all others roots of its minimal polynomial have absolute 
value less than $1$. Similarly, if $n>m$, then $c^{-1}$ is a Pisot number. More precisely, these 
Pisot numbers belong to the class $\mathcal P$ considered in \cite{BH}.
\end{rems}

\medskip

\begin{ejemplo}
For $m=0,\, n=3$, it can be seen that the possible values of $b$ are
\[ b= \left| \frac{\arctan 
\left(\frac{\sqrt{3}(r^2+1)}{r^2-1}\right)+2k\pi}{\log(r-r^{-1})}\right|,\quad 
r=\left(\frac{1+\sqrt{5}}{2} \right)^\frac13, \quad k\in\mathbb Z. 
\]
\end{ejemplo}

\

\section{Lattices in almost abelian Lie groups with LCS structures}

In this section we begin by characterizing locally conformal symplectic structures on almost 
abelian Lie algebras. Then we study the existence of lattices in the associated simply connected 
almost abelian Lie groups: in dimension $4$ we determine all these Lie groups, and in each 
higher dimension we exhibit an almost abelian Lie group with a countable family of non isomorphic 
lattices.

\subsection{LCS almost abelian Lie algebras}

Let $(\omega,\theta)$ denote an LCS structure on an almost abelian Lie algebra $\g$ of
dimension $2n+2$. If $\u\subset \g$ is an abelian ideal of $\g$ of codimension $1$, then
$\omega|_{\u \times \u}$ has rank $2n$. Therefore there exist $f_2\in\u$ and a 
vector subspace $\vg\subset \u$ such that $\u=\r f_2\oplus \vg$, $\omega(f_2,\vg)=0$ and 
$\omega|_{\vg \times \vg}$ is non degenerate. Since $\omega$ is non degenerate on $\g$, there 
exists $f_1\in\g$ such that $\omega(f_1,f_2)=1$ and 
\begin{equation}\label{g=r.u}
\g=\r f_1\oplus \u. 
\end{equation}

The adjoint action of $f_1$ on $\u$ is given by 
\begin{equation}\label{mu}
[f_1,f_2]=\mu f_2+ v_0, 
\end{equation}
for some $v_0\in\vg$ and $\mu\in\r$, and for $x\in \vg$ we have 
\[ [f_1,x]=\eta(x) f_2+ Ax, \quad \text{with } \eta\in \vg^*, \, A\in\operatorname{End}(\vg).\]

\medskip

Since $\omega$ is non degenerate on $\vg$, there exists a basis $\{u_1,\dots,u_n,v_1,\dots,v_n\}$ of 
$\vg$ such that 
\[ \w=f^1\wedge f^2 + \sum_{i=1}^n u^i\wedge v^i, \] 
where $\{f^1, f^2, u^1,\dots,u^n,v^1,\dots,v^n\}$ is the dual basis of $\g^*$. We can identify $\vg$ 
with $\r^{2n}$, $\u$ with $\r^{2n+1}$, $\g$ with $\r\ltimes\r^{2n+1}$, and if we denote 
$M=\ad_{f_1}|_{\r^{2n+1}}$, then we can express $M$ as
\begin{equation}\label{M}
M=\left(\begin{array}{c|ccc}       
\mu &   & w^t & \\
\hline
  &   &   & \\
v_0 &   & A & \\
  &   &   & 
\end{array}\right),
\end{equation}
where $w^t=(\eta(u_1),\ldots,\eta(u_n),\eta(v_1), \ldots, \eta(v_n))$.

\medskip

We will consider two different cases, depending on whether the dimension of $\g$ is $4$ or greater
than $4$. Unlike the locally conformal K\"ahler case, the description of the Lie algebras
admitting LCS forms in dimension $4$ will be different from those Lie algebras in higher dimensions.

\

\subsubsection{Dimension of $\g\geq 6$} 

In this case, we have that $\dim\vg\geq 4$. We will see next that the corresponding Lee 
form $\theta$ is given by $\theta=a f^1$ for some $a\in\r$, $a\neq 0$, and the vector $v_0$ 
in \eqref{mu} is $v_0=0$.

Since $\omega$ is non degenerate on $\vg$, for any nonzero $x\in\vg$, there exist $x',y,y'\in 
\vg$ such that $\omega(x,y)=\omega(x,y')=0$ and $\omega(x,x')=\omega(y,y')=1$. Then:
\begin{itemize}
 \item[$\diamond$] $d\omega(x,y,y')=\theta\wedge \omega(x,y,y')$ implies that $\theta(x)=0$, 
so that $\theta|_{\vg}=0$;
 \item[$\diamond$] $d\omega(f_2,x,x')=\theta\wedge \omega(f_2,x,x')$ and the fact that 
$\theta|_{\vg}=0$ imply that $\theta(f_2)=0$. 
\end{itemize}
Thus $\theta=a f^1$ for some $a\in\r$, $a\neq 0$.

We compute next $d\omega(f_1,f_2,x)=\theta\wedge \omega(f_1,f_2,x)$. We have that 
\begin{align*}
d\omega(f_1,f_2,x) & = -\omega([f_1,f_2],x)-\omega([x,f_1],f_2)\\
                  & = -\omega(\mu f_2+v_0,x) + \omega(\eta(x)f_2+Ax,f_2)\\
                  & = -\omega(v_0,x)                                     
\end{align*}
since $\omega(f_2,\vg)=0$ and $\u$ is abelian.
On the other hand, $\theta\wedge \omega(f_1,f_2,x)=\theta(f_1)\omega(f_2,x)=0$ since $\theta=af^1$. 
As a consequence, $\omega(v_0,x)=0$ for all $x\in\vg$. Since $\omega$ is non degenerate on $\vg$, 
it follows that $v_0=0$.

\medskip

Next we take into account $d\omega(f_1,x,y)=\theta\wedge \omega(f_1,x,y)$ for $x,y\in\vg$. Since 
$\theta=af^1$, the right-hand side is $\omega(f_1,x,y)=a\,\omega(x,y)$. On the other hand,
\begin{align*}
 d\omega(f_1,x,y) & = -\omega([f_1,x], y)-\omega([y,f_1],x)\\
                  & = -\omega(\eta(x)f_2+Ax,y)-\omega(x,\eta(y)f_2+Ay)\\
                  & = -\omega(Ax,y)-\omega(x,Ay),
\end{align*}
since $f_2$ is $\omega$-orthogonal to $\vg$. Let us decompose $A=U+B$, with $U^{\ast_\omega}=U$ and 
$B^{\ast_\omega}=-B$. Recall that given a linear transformation $T$ on the symplectic vector space 
$(\vg,\omega|_{\vg\times \vg})$, its $\omega$-adjoint $T^{\ast_\omega}$ is defined by 
$\omega(Tu,v)=\omega(u,T^{\ast_\omega}v)$ for any $u,v\in \vg$. Note that $B^{\ast_\omega}=-B$ 
means that $B\in \mathfrak{sp}(\vg,\omega|_{\vg\times \vg})\simeq \mathfrak{sp}(n,\r)$. 

With this decomposition the equation above becomes $d\omega(f_1,x,y)=-2\omega(Ux,y)$, hence
\[ -2\omega(Ux,y)=a\,\omega(x,y) \quad \text{for all } x,y\in\vg.\]
It follows that $U=-\frac{a}{2}\operatorname{Id}$ and
\[ A=-\frac{a}{2}\operatorname{Id}+B, \quad \text{with } B\in \mathfrak{sp}(n,\r).\]
Therefore $M$ has the following matrix representation with respect to the basis above:
\[M=\left(\begin{array}{c|ccc}       
\mu &   & w^t & \\
\hline
   &   &   & \\
 0 &   & -\frac{a}{2}I + B & \\
   &   &   & 
\end{array}\right),\]
with $B\in \mathfrak{sp} (n,\mathbb R)$. Summarizing, we have the following result:

\medskip

\begin{teo}\label{lcs}
Let $\g$ be an $(2n+2)$-dimensional almost abelian Lie algebra with $\dim\g \geq 6$.  
Then $\g$ admits an LCS structure if and only if $\g$ is isomorphic to $\r\ltimes_M\r^{2n+1}$, 
where the adjoint action of $\r$ on $\r^{2n+1}$ is given by
\begin{equation} \label{lcs-6}
 M=\left(\begin{array}{c|ccc}       
\mu &   & w^t & \\
\hline
   &   &   & \\
 0 &   & \lambda I + B & \\
   &   &   & 
\end{array}\right),
\end{equation}
for some $\mu,\lambda\in \r$, $\lambda\neq 0$, $w\in\r^{2n}$ and $B\in \mathfrak{sp}(n,\r)$. The 
LCS form is $\omega=f^1\^f^2 + \displaystyle\sum_{i=1}^{n} u^i\^v^i$, and the Lee form is 
$\theta=-2\lambda f^1$. Moreover, $\g$ is unimodular if and only if $\lambda=-\frac{\mu}{2n}$. 
\end{teo}

\medskip

\begin{rems}
(i) In the same way, it can be seen that any symplectic almost abelian Lie algebra is isomorphic to 
$\r\ltimes_M\r^{2n+1}$ with $M$ as in \eqref{lcs-6} with $\lambda=0$ (cf. \cite{LW}).

\smallbreak\n(ii) When $w=0$ and $B\in\mathfrak{u}(n)\subset\mathfrak{sp}(n,\r)$, the LCS form
is in fact LCK. Indeed, in the notation of the theorem, the almost complex structure $J$ defined 
by $Jf_1=f_2,\, Ju_i=v_i,\, i=1,\ldots, n$, is integrable and the metric 
$\pint=\omega(\cdot,J\cdot)$ is compatible with $J$. 
\end{rems}

\medskip

In the following result we show that the LCS structures constructed in Theorem \ref{lcs} are all of 
the second kind. Therefore this theorem provides a way of producing many examples of Lie algebras 
with this type of structures.

\smallskip

\begin{cor}\label{2kind}
Let $\g$ be an almost abelian Lie algebra with $\dim\g \geq 6$. If $\g$ admits an LCS structure, 
then it is of the second kind.
\end{cor}

\begin{proof}
Let $(\omega, \theta)$ be the LCS structure on $ \g$. From Theorem \ref{lcs} we have that 
$\theta=-2\lambda 
f^1$. Recall that the LCS structure is of the second kind if $\theta|_{\g_\omega}\equiv 0$, where $\g_\omega$ is given by \eqref{autom}. 

Let $x\in \g_\omega$, $x=cf_1+v$ with $c\in \r$ and $v\in \r^{2n+1}$. We compute 
$d\omega(x,y,z)=\theta\wedge\omega(x,y,z)$ for $y,z\in\r^{2n+1}$.
Since $x\in\g_\omega$ we have $c\lambda\,\omega(y,z)=0$ for all $y,z\in\r^{2n+1}$. Since $\omega$ 
is non degenerate and $\lambda\neq0$ we obtain $c=0$. Therefore $\theta|_{\g_\omega}$ is identically zero, hence the LCS structure is of the second kind.
\end{proof}

\

\subsubsection{Dimension of $\g=4$}

In this section we will proceed in a different way to determine the $4$-dimensional almost abelian 
Lie algebras admitting an LCS structure. 

Let $\g$ be a $4$-dimensional almost abelian Lie algebra equipped with an LCS structure, 
$\g=\r\ltimes_M\r^3$, where $M$ denotes the adjoint action of $\r$ on $\r^3$, and let $\mu$ be a 
real eigenvalue of $M$. Then we have
\[
M=\left(\begin{array}{c|ccc}       
\mu &   & w^t & \\
\hline
   &   &   & \\
 0 &   & A & \\
   &   &   & 
\end{array}\right),
\]
where $A\in \mathfrak{gl} (2,\mathbb R)$ for some basis $\{f_1,f_2,f_3,f_4\}$.  

\begin{lema}\label{M0}
With notation as above, if $\tr(A)\neq 0$ then $\g$ admits an LCS structure.
\end{lema}

\begin{proof}
It is easy to see that $\omega=f^1\^f^2 +f^3\^f^4$ and $\theta=-\tr(A)f^1$ satisfy 
$d\omega=\theta\^\omega$ and $d\theta=0$, where $\{f^1,f^2,f^3,f^4\}$ is the dual basis of $\g^*$.
\end{proof}

Given $\g=\r\ltimes_M\r^3$, according to Lemma \ref{ad-conjugated}, we may assume that $M$ is in 
its canonical Jordan form, up to scaling. In this case, there are four different 
possibilities for $M$:
\[M_1=\begin{pmatrix}       
\lambda_1 &   &  \\
   & \lambda_2  &   \\
  &   &  \lambda_3 
\end{pmatrix}, \,
M_2^{\mu,\lambda}=\begin{pmatrix}       
\mu &   &  \\
   & \lambda  & 1  \\
  &   &  \lambda 
\end{pmatrix}, \,
M_3^\mu=\begin{pmatrix}       
\mu &  1 &  \\
   & \mu  & 1  \\
  &   &  \mu 
\end{pmatrix}, \,
M_4^{\mu,\lambda}=\begin{pmatrix}       
\mu &   &  \\
   & \lambda  & -1  \\
  & 1  &  \lambda 
\end{pmatrix},\]
for $\lambda, \mu, \lambda_i\in \r$, $\lambda_1^2+\lambda_2^2+\lambda_3^2\neq 0$.

\medskip

The only cases that are not covered by Lemma \ref{M0} are the following:
\[M_2^{0,0}= \begin{pmatrix}       
0 &   &  \\
   & 0  & 1  \\
  &   &  0 
\end{pmatrix}, \quad
M_3^0=\begin{pmatrix}       
0 &  1 &  \\
   & 0  & 1  \\
  &   &  0 
\end{pmatrix}, \quad
M_4^{\mu,0}=\begin{pmatrix}       
\mu &   &  \\
   & 0  & -1  \\
  & 1  &  0 
\end{pmatrix}.\]


By a direct computation it can be seen that $\g=\r\ltimes_{M}\r^3$ admits LCS structures for $M=M_2^{0,0},\,M_3^0, \,M_4^{0,0}$, whereas for $M=M_4^{\mu,0}$ with $\mu\neq 0$ it does not admit any. We can summarize these results in the next theorem, where we use the notation from \cite{ABDO}.

\medskip

\begin{teo}\label{4-lcs}
Let $\g$ be a $4$-dimensional almost abelian Lie algebra with an LCS structure. Then $\g$ 
is isomorphic to one of the following Lie algebras:  

\smallskip

$\begin{aligned}
\mathfrak h_3\times\mathbb R: & \, [e_1,e_2]=e_3 \\
\mathfrak n_4: & \, [e_1,e_2]=e_3 , [e_1,e_3]=e_4 \\
\mathfrak r_{3,\lambda}\times\mathbb R: & \, [e_1,e_2]=e_2, [e_1,e_3]=\lambda e_3\\ 
\mathfrak r_{4,\mu,\lambda}: & \, [e_1,e_2]=e_2, [e_1,e_3]= \mu e_3, [e_1,e_4]=\lambda e_4, \; \mu\lambda\neq0 \\
\mathfrak r_3\times\mathbb R: & \, [e_1,e_2]=e_2, [e_1,e_3]= e_2 + e_3  \\
\mathfrak r_{4,\lambda}: & \, [e_1,e_2]=e_2, [e_1,e_3]=\lambda e_3, [e_1,e_4]= e_3 + \lambda e_4 \\
\mathfrak r_4: & \, [e_1,e_2]=e_2, [e_1,e_3]=e_2 + e_3, [e_1,e_4]=e_3 + e_4 \\
\mathfrak r'_{3,\lambda}\times\mathbb R: & \, [e_1, e_2]=\lambda e_2 -e_3, [e_1, e_3]=e_2 + \lambda e_3 \\
\mathfrak r'_{4,\mu,\lambda}: & \, [e_1,e_2]=\mu e_2, [e_1,e_3]=\lambda e_3-e_4, [e_1,e_4]=e_3+\lambda e_4, \, \mu\neq0,\, \lambda\neq0 \\
\end{aligned}$
\end{teo}

\

\begin{proof}
We prove first that, as mentioned above, $\g=\r f_1\ltimes_M\r^3$ does not admit LCS structures for $M=M_4^{\mu,0}$ with $\mu\neq 0$. Let $\{f^1,f^2,f^3,f^4\}$ be the dual basis of $\g^*$, and let us assume that $(\omega,\theta)$ is an LCS structure on $\g$, where
\begin{align*}
\omega & =a_1f^1\^f^2 +a_2f^1\^f^3 +a_3f^1\^f^4 +a_4f^2\^f^3 +a_5f^2\^f^4 +a_6f^3\^f^4, \\ 
\theta & =b_1f^1+b_2f^2+b_3f^3+b_4f^4,
\end{align*}
for some $a_i, b_j\in\r$. Since $\theta$ is closed, we have that $b_2=b_3=b_4=0$.
Computing $d\omega=\theta\^\omega$ and using that $\theta\neq0$ we have 
$a_6=0$, $b_1a_4=-\mu a_4+a_5$ and $b_1a_5=-a_4-\mu a_5.$
Therefore $a_4=a_5=0$, which is a contradiction with the fact that $\omega$ is non degenerate.

\smallskip

To finish the proof we refer to Table \ref{1y2kind}, where LCS structures for all the Lie algebras in the statement are exhibited.
\end{proof}

\medskip

\begin{rems}\label{lcs-unimod} (i) The Lie algebra $\g=\r\ltimes_{M_4^{0,0}}\r^3$ corresponds to 
the Lie algebra $\mathfrak r'_{3,0}\times\r$, while $\g=\r\ltimes_{M_4^{\mu,0}}\r^3$, $\mu\neq0$, 
corresponds to $\mathfrak r'_{4,\mu,0}$.

\smallbreak\n(ii) Among the Lie algebras in Theorem \ref{4-lcs}, the unimodular ones are:  
$\mathfrak h_3\times\mathbb R$, $\mathfrak n_4$, $\mathfrak r_{3,-1}\times\mathbb R$, 
$\mathfrak r_{4,\lambda,-(1+\lambda)}$, $\mathfrak r_{4,-\frac12}$, $\mathfrak r'_{3,0}\times\mathbb R$ 
and $\mathfrak r'_{4,\lambda,-\lambda/2}$.

\smallbreak\n(iii) According to \cite{Ov}, the Lie algebras $\mathfrak r_{3,\lambda}\times\mathbb R$ 
($\lambda\neq 0,-1$), $\mathfrak r_{4,\mu,\lambda}$ ($\mu\neq-\lambda$ or $\mu\neq-1$), 
$\mathfrak r_3\times\mathbb R$, $\mathfrak r_{4,\lambda}$ ($\lambda\neq0,-1$),
$\mathfrak r_4$, $\mathfrak r'_{3,\lambda}\times\mathbb R$ ($\lambda\neq0$) and
$\mathfrak r'_{4,\mu,\lambda}$ do not admit any symplectic structures.

\smallbreak\n(iv) Note that $\mathfrak r_{3,-1}$ is the Lie algebra $\mathfrak e(1, 1)$ (it is also
denoted by $\mathfrak{sol}^3$) of the group of rigid motions of Minkowski $2$-space, and
$\mathfrak r'_{3,0}$ is the Lie algebra $\mathfrak e(2)$ of the group of rigid motions of Euclidean
$2$-space. 
\end{rems}

\bigskip

Next, we will study whether the LCS structures on these $4$-dimensional almost abelian Lie algebras 
are of first or second kind. We prove first a result which holds in any dimension.

\begin{lema}
Let $\g=\r f_1\ltimes_M\u$ be an almost abelian Lie algebra as in $\eqref{g=r.u}$ endowed with an 
LCS structure $(\omega, \theta)$. If $M$ is invertible then such LCS structure is of the second 
kind. 
\end{lema}

\begin{proof}
Indeed, if $M$ is invertible then $[\g,\g]=\u$. Since $\theta([\g,\g])=0$, it follows that 
$\theta=cf^1$ for some $c\in\r,\, c\neq0$. Let $z=af_1+v\in\g_\omega$ with $a\in\r, v\in\u$. Computing $d\omega(z,x,y)=\theta\^\omega(z,x,y)$ for all $x,y\in\u$, we obtain that $a=0$ and therefore $\g_\omega\subset\u$.

\end{proof}

Let $\g=\r\ltimes_M\r^{3}$ be a $4$-dimensional almost abelian Lie algebra equipped with an LCS 
structure. The cases included in the lemma above correspond to: $\mathfrak r_{4,\mu,\lambda}$, 
$\mathfrak r_{4,\lambda}$ ($\lambda\neq0$), $\mathfrak r_4$ and $\mathfrak r'_{4,\mu,\lambda}$, 
$\lambda\neq0$. Therefore any LCS structure on these Lie algebras is of the second kind.

Let us suppose  now that $M$ is not invertible. It was proved in \cite{BM} that any LCS structure on 
a nilpotent Lie algebra is of the first kind, so that we may assume that $M$ is not nilpotent. The 
Jordan form of such a matrix can be one of the following:
\[ M_1=\begin{pmatrix}       
\lambda_1 &   &  \\
   & \lambda_2  &   \\
  &   &  \lambda_3 
\end{pmatrix} \text{with} \,\,\, \lambda_1\lambda_2\lambda_3=0, \qquad
M_2^{0,\lambda}=\begin{pmatrix}       
0 &   &  \\
   & \lambda  & 1  \\
  &   &  \lambda
\end{pmatrix} \text{with} \,\,\, \lambda\neq0, \]
\[ M_2^{\mu,0}=\begin{pmatrix}       
\mu &   &  \\
   & 0  & 1  \\
  &   &  0 
\end{pmatrix} \text{with} \,\,\, \mu\neq0, \qquad
M_4^{0,\lambda}=\begin{pmatrix}       
0 &   &  \\
 & \lambda  & -1  \\
 & 1  &  \lambda 
\end{pmatrix}.\]

By direct computations we can verify that all these Lie algebras admit both LCS structures of 
the first kind and of the second kind, except for the case $M_4^{0,0}$, which 
admits only LCS structures of the first kind. This Lie algebra corresponds to $\mathfrak 
r'_{3,0}\times\r$ in Theorem \ref{4-lcs}.

Let us show the last statement, i.e. the almost abelian Lie algebra $\mathfrak r'_{3,0}\times\r$ 
admits LCS structures only of the first kind. Indeed, let us suppose that 
$\omega=a_1f^{12}+a_2f^{13}+a_3f^{14}+a_4f^{23}+a_5f^{24}+a_6f^{34}$, 
$\theta=b_1f^1+b_2f^2+b_3f^3+b_4f^4$ is an LCS structure on this Lie algebra with $a_i, b_j\in\r$. 
From $d\omega=\theta\^\omega$ we obtain that $b_2\neq0$. On the other hand, 
$f_2\in\mathfrak{z}(\g)$, then $\textrm{L}_{f_2}\omega=0$. Thus $f_2\in\g_\omega$ and 
$\theta(f_2)=b_2\neq0$. Therefore $(\theta, \omega)$ is of the first kind.

\medskip

We can summarize these results in Table $\ref{1y2kind}$, where we exhibit LCS structures of 
the first or second kind for each Lie algebra. The empty spaces mean that the corresponding 
Lie algebra does not admit any LCS structure of that specific kind.

\begin{table}[h]
\begin{center}
\renewcommand{\arraystretch}{1.2}
\begin{tabular}{|c|c|c|}\hline

Lie algebra & LCS first kind &  LCS second kind  \\ \hline

$\mathfrak h_3\times\mathbb R$ & $\omega=e^1\wedge e^2+e^3\wedge e^4$ & \\
& $\theta=-e^4$ & \\ 
\hline

$\mathfrak n_4$ & $\omega=e^1\wedge e^3+e^2\wedge e^4$ & \\
& $\theta=e^2$ &\\ 
\hline

$\mathfrak r_{3,\lambda}\times\mathbb R$ & $\omega=2e^1\wedge e^2+(1+\lambda) e^1\wedge e^3-e^2\wedge e^4-e^3\wedge e^4$ & $\omega=e^1\wedge e^4-e^2\wedge e^3$ \\
($\lambda\neq-1$) & $\theta=e^1+e^4$ & $\theta=-(1+\lambda)e^1$ \\ 
\hline

$\mathfrak r_{3,-1}\times\mathbb R$ & $\omega=2e^1\wedge e^2-e^2\wedge e^4-e^3\wedge e^4$ & $\omega=e^1\wedge e^2-e^3\wedge e^4$ \\
 & $\theta=e^1+e^4$ & $\theta=e^1$ \\
\hline

$\mathfrak r_{4,\mu,\lambda}$ & & $\omega=e^1\wedge e^3+e^2\wedge e^4$ \\
$(\lambda\neq -1)$ & & $\theta=-(\lambda+1)e^1$ \\ \hline

$\mathfrak r_{4,\mu,-1}$ & & $\omega=e^1\wedge e^2+e^3\wedge e^4$ \\
($\mu\neq 1$) & & $\theta=(1-\mu)e^1$, \\  \hline

$\mathfrak r_{4,1,-1}$ & & $\omega=e^1\wedge e^4+e^2\wedge e^3$ \\
& & $\theta=-2e^1$ \\ \hline

$\mathfrak r_3\times\mathbb R$ & $\omega=e^1\wedge e^2+2e^1\wedge e^3-e^2\wedge e^4-e^3\wedge e^4$ & $\omega=e^1\wedge e^4-e^2\wedge e^3$ \\
& $\theta=e^4$ & $\theta=-2e^1$ \\ \hline

$\mathfrak r_{4,\lambda}$ & & $\omega=e^1\wedge e^2-e^3\wedge e^4$ \\
($\lambda\neq0$) & & $\theta=-2\lambda e^1$ \\ \hline

$\mathfrak r_{4,0}$ & $\omega=-e^1\wedge e^2+e^2\wedge e^4+e^3\^e^4$ & $e^1\wedge e^3+e^2\wedge e^4$ \\
& $\theta=e^4$ & $\theta=-e^1$\\ \hline

$\mathfrak r_4$ & & $\omega=e^1\wedge e^2+e^3\wedge e^4$\\
& & $\theta=-2e^1$ \\ \hline

$\mathfrak r'_{3,\lambda}\times\mathbb R$ & $\omega=-e^1\^e^2+2\lambda e^1\wedge e^3-e^3\wedge e^4$ & $\omega= e^1\wedge e^4-e^2\wedge e^3$\\
($\lambda\neq0$)  & $\theta=\lambda e^1+e^4$ & $\theta=-2\lambda e^1$ \\ \hline

$\mathfrak r'_{3,0}\times\mathbb R$ & $\omega=e^1\wedge e^2+e^3\wedge e^4$ & \\
 & $\theta=e^4$ &\\ \hline

$\mathfrak r'_{4,\mu,\lambda}$ & & $\omega=e^1\wedge e^2+e^3\wedge e^4$ \\
($\lambda\neq0$) & & $\theta=-2\lambda e^1$ \\ \hline

\end{tabular} 
\bigskip
\caption{LCS structures of the first or second kind on $4$-dimensional Lie algebras 
}\label{1y2kind}
\end{center}
\end{table}

\

\subsection{Lattices in the associated LCS Lie groups}

\subsubsection{Lattices in dimension $4$}

In this section we will determine up to Lie algebra isomorphism the almost abelian Lie algebras in 
Theorem \ref{4-lcs} whose associated simply connected Lie groups admit lattices. In order
to do this we use the classification of unimodular completely solvable Lie groups of type $\mathbb
R\ltimes \mathbb R^3$ in \cite{LLSY} and Proposition \ref{latt}. 

\begin{teo}\label{4-lcs-latt}
Let $G$ be a simply connected $4$-dimensional unimodular almost abelian Lie group with a 
left invariant LCS structure, and let
$\g$ denote its Lie algebra. If $G$ admits lattices then $\g$ is isomorphic to one of the following 
Lie algebras:
$\mathfrak h_3\times\mathbb R, \, \mathfrak n_4,  \, \mathfrak r_{3,-1}\times\mathbb R, \, \mathfrak 
r'_{3,0}\times\mathbb R, \, \mathfrak r_{4,\lambda,-(1+\lambda)}$ for countably many values of 
$ \lambda >1$, or $\mathfrak r'_{4,\lambda,-\lambda/2}$ for countably many values of $\lambda >0$.
\end{teo}

\begin{rem}
 It is easy to verify that if $\lambda'\in \{\lambda, -(1+\lambda), \frac1\lambda, 
-\frac{1}{1+\lambda},  -\frac{\lambda}{1+\lambda}, -\frac{\lambda+1}{\lambda}\}$ then $\mathfrak 
r_{4,\lambda',-(1+\lambda')}$ is isomorphic to $\mathfrak r_{4,\lambda,-(1+\lambda)}$; therefore we 
may consider $\lambda\geq 1$ for this family. In the same way, it is readily checked that 
$\mathfrak r'_{4,\lambda',-\lambda'/2}$ is isomorphic to $\mathfrak r'_{4,\lambda,-\lambda/2}$ if 
$\lambda'=\pm \lambda$, so that in this case we may assume $\lambda>0$.
\end{rem}

\begin{proof}
Let us consider the unimodular Lie algebras in Theorem \ref{4-lcs}, see Remark \ref{lcs-unimod}(ii).

In the nilpotent case, the simply connected nilpotent Lie groups corresponding to the Lie algebras
$\mathfrak h_3\times\mathbb R$ and $\mathfrak n_4$ admit lattices due to Malcev's criterion, since 
these Lie algebras have rational structure constants for some basis.

Next, we consider the case of the completely solvable non-nilpotent Lie algebras $\mathfrak 
r_{3,-1}\times\mathbb R$, $\mathfrak r_{4,\lambda,-(1+\lambda)}$ $(\lambda\geq1)$ and $\mathfrak 
r_{4,-1/2}$. It is well known that the simply connected Lie group $Sol^3$ corresponding to 
$\mathfrak r_{3,-1}\simeq \mathfrak{sol}^3$ admits lattices (see for instance \cite{MR}), and
therefore $Sol^3\times \mathbb R$ admits lattices as well. Now we use the classification given in 
\cite{LLSY}, noting that $\mathfrak r_{4,\lambda,-(1+\lambda)}\simeq \mathfrak{sol}^4_\lambda$ 
$(\lambda>1)$, $\mathfrak r_{4,1,-2}\simeq \mathfrak{sol}^4_0$ and $\mathfrak r_{4,-1/2}\simeq 
\mathfrak{sol}'^4_0$. The simply connected Lie group associated to $\mathfrak{sol}^4_\lambda$ admits 
lattices for countably many $\lambda's$ (see \cite[Proposition 2.1]{LLSY}), whereas the Lie groups 
$Sol^4_0$ and $Sol'^4_0$ associated to $\mathfrak{sol}^4_0$ and $\mathfrak{sol}'^4_0$ do not admit 
any lattices \cite[Proposition 2.2]{LLSY}. Note that $Sol^4_0=G_0$ from \S \ref{lat-dim-4}.

Finally, we take into account the non-completely solvable Lie algebras $\mathfrak 
r'_{3,0}\times\mathbb R$ and $\mathfrak r'_{4,\lambda,-\lambda/2}$ $(\lambda>0)$. It is well known 
that the simply connected Lie group $E(2)$ corresponding to $\mathfrak r'_{3,0}\simeq 
\mathfrak{e}(2)$ admits lattices (this fact is also an easy application of Proposition \ref{latt}), 
and therefore $E(2) \times \mathbb R$ admits lattices, too. The Lie algebra $\mathfrak
r'_{4,\lambda,-\lambda/2}$ $(\lambda>0)$ is isomorphic to $\g_{b}$ from \S \ref{lat-dim-4}, 
for $b=\frac{1}{\lambda}$, and it has already been proved in Theorem \ref{4-lattice} that the 
corresponding simply connected Lie groups $G_b$ admit lattices for countably many values of the 
parameter $b>0$.
\end{proof}

\smallskip

\begin{rem}
The Lie algebras $\mathfrak n_4$, $\mathfrak r_{3,-1}\times\mathbb R$, $\mathfrak 
r_{4,\lambda,-(1+\lambda)}$ for countably many values of $\lambda >1$ and $\mathfrak 
r'_{3,0}\times\mathbb R$ provide examples of solvmanifolds with an LCS structure which carry no LCK 
structure (coming from a left invariant LCK structure on the Lie group). We point out that in 
\cite{BM} it was shown that a nilmanifold associated to $\mathfrak n_4$ admits an LCS structure of 
the first kind but does not admit any LCK metric. Moreover, this nilmanifold is not the product of 
a $3$-dimensional compact manifold and a circle.
\end{rem}

\medskip

\subsubsection{Lattices in dimension $\geq 6$}

In this subsection we will provide examples of almost abelian solvmanifolds with an LCS structure. 
For $n\geq 2$, let $\g$ be the $(2n+2)$-dimensional unimodular almost abelian Lie algebra given by 
$\g=\mathbb Rf_1\ltimes_M\mathbb 
R^{2n+1}$ with 
\begin{equation}\label{M6}
M=\left(\begin{array}{c|ccccccccc}       
1 &  & & & & & & & \\
\hline
 & 0 & & & & & & &\\
 & & \frac{1}{n} & & & & & &\\
 & & & \ddots & & & & &\\
 & & & & \frac{n-1}{n} & & & &\\
 & & & & & -\frac{1}{n}  & & &\\
 & & & & & & -\frac{2}{n}  &  \\
 & & & & & & & \ddots & \\
 & & & & & & & & -1
\end{array}\right),\end{equation}
in a basis $\{f_2, u_1,\dots,u_n,v_1,\dots,v_n\}$ of $\r^{2n+1}$. Note that $M$ is given 
by \eqref{M6} with 
\begin{gather*} 
\mu=1, \quad  \lambda=-\frac{1}{2n}, \quad w=0, \\
 B=\operatorname{diag}\left(\frac{1}{2n}, 
\frac{3}{2n},\ldots, \frac{2n-1}{2n},-\frac{1}{2n}, -\frac{3}{2n},\ldots,-\frac{2n-1}{2n} 
\right)\in\mathfrak{sp}(n,\r).
\end{gather*}
According to Theorem \ref{lcs} the Lie algebra $\g$ admits an LCS form $\omega=f^1\^f^2 + 
\sum_{i=1}^{n} u^i\^v^i$ with Lee form $\theta=\frac 1n f^1$. From Corollary \ref{2kind} we have 
that this LCS structure is of the second kind.

The simply connected Lie group associated to $\g$ is $G=\mathbb R\ltimes_\phi\mathbb R^{2n+1}$ where 
$\phi$ is given by 
\[\phi(t)=e^{tM}=\left(\begin{array}{c|ccccccccc}       
e^t &  & & & & & & &  \\
\hline
 & 1 & & & & & & &\\
 & & e^{\frac{t}{n}} & & & & & &\\
 & & & \ddots & & & & &\\
 & & & & e^{\frac{(n-1)t}{n}} & & & &\\
 & & & & & e^{-\frac{t}{n}}  & & &\\
 & & & & & & e^{-\frac{2t}{n}} & &\\
 & & & & & & & \ddots & \\
 & & & & & & & & e^{-t}
\end{array}\right).\]

The characteristic polynomial of $\phi(t)$ is 
\[ p(x)=(x-1)(x-\rho^2)(x-\rho^{-2})(x-\rho^4)(x-\rho^{-4})\dots(x-\rho^{2n})(x-\rho^{-2n}), \]
where $\rho=\displaystyle{e^{\frac{t}{2n}}}$. Fixed $m\in\NN, m>2$, we define 
\[t_m=n\operatorname{arccosh}\left(\frac m2\right), \quad t_m>0.\]
Then $\rho_m=\displaystyle{e^{\frac{t_m}{2n}}}$ satisfies $\rho_m^2 + \rho_m^{-2}=m.$ 

We define the following real sequence $a_0=2$, $a_1=m$, $a_k=\rho_m^{2k} + \rho_m^{-2k}$ for 
$k=3,\dots,n$. It is easy to see that $a_k$ satisfies $a_{k+1}=ma_k-a_{k-1}$ and therefore 
$a_k\in\ZZ$ for all $k$. As a consequence we have that 
$(x-\rho_m^{2k})(x-\rho_m^{-2k})=x^2-a_kx+1\in \ZZ[x]$, thus we can write
\[p(x)=x^{2n+1}-m_{2n}x^{2n}+m_{2n-1}x^{2n-1}+ \dots +m_1 x-1\] for some 
$m_1,m_2,\dots,m_{2n}\in\ZZ$. Since all the roots of $p$ are different, we have that $\phi(t_m)$ is 
conjugated to the companion matrix $B$ of $p$, that is,
\begin{equation}\label{companion}
B_m=\begin{pmatrix}       
 0 & & & & 1 \\
 1 & 0 & & & -m_1 \\
 & \ddots&\ddots & & \vdots \\
 & & 1 & 0 & -m_{2n-1}\\
 & & & 1 & m_{2n}
\end{pmatrix}.
\end{equation}
According to Proposition \ref{latt}, the group $G$ admits lattices. One such lattice is 
given by
\[\Gamma_m := t_m\mathbb Z\ltimes_\phi P_m^{-1}\mathbb Z^{2n+1},\] 
for each $m>2$, where $P_m$ satisfies $P_m\phi(t_m)P_m^{-1}=B_m$. Therefore for all $m>2$, 
$M_m:=\Gamma_m\backslash G$ is a solvmanifold with an LCS structure. 


\medskip

\begin{prop}
With notation as above, the solvmanifolds $M_m$ with $m\in\NN, m>2$, are pairwise non 
homeomorphic. 
\end{prop}

\begin{proof}
Let us assume that $M_r$ and $M_s$, with $r,s>2$ are homeomorphic. Then $\pi_1(M_r)$ is isomorphic 
to $\pi_1(M_s) $. Since $G$ is simply connected, we have that the fundamental groups of these 
solvmanifolds are isomorphic to the lattices $\Gamma_r$ and $\Gamma_s$, so that these lattices are 
isomorphic. Since $G$ is completely solvable, the Saito's rigidity theorem \cite{Sai} implies 
that this isomorphism extends to a Lie group automorphism of $G$. Since the lattices differ by an 
automorphism of $G$, it follows from \cite[Theorem 2.5]{Hu} that the integer matrix $B_r$ (as in 
\eqref{companion}) is conjugated either to $B_s$ or to its inverse in $GL(n,\mathbb{Z})$. It can be 
seen that this happens if and only if $r=s$.
\end{proof}

\medskip

\begin{rem}
It is easy to see that the LCS solvmanifolds $M_m$ do not admit invariant 
complex structures. However, they do admit invariant symplectic structures, for instance, 
$\eta= f^1\wedge u^1+f^2\wedge v^n +\sum_{i=1}^{n-1} u^{i+1}\wedge v^i$.
\end{rem}

\

\subsubsection{De Rham and adapted cohomology of the LCS solvmanifolds $\Gamma_m \backslash G$}
We begin with the study of the de Rham cohomology of the solvmanifold $M_m=\Gamma_m\backslash G$. 
Since $G$ is completely solvable, it follows from \eqref{deRham} that $H^k_{dR}(M_m)$ is isomorphic to 
the Lie algebra cohomology group $H^k(\g)$. Therefore, we need to study the Chevalley-Eilenberg 
complex of the almost abelian Lie algebra $\g=\mathbb Rf_1\ltimes_M\mathbb R^{2n+1}$ with $M$ as in 
\eqref{M6}. After a change of basis in the ideal $\u=\r^{2n+1}$, we can assume that $M$ is given by
\[ M=\operatorname{diag}\left(0, \frac1n,\frac2n,\ldots, 1, -\frac1n,-\frac2n,\ldots, -1\right). \]

According to \cite{Sa}, the $k^{th}$ Betti number of $\g$, $\beta^n_k=\dim H^k(\g)$, can be computed 
as follows, where $Z^j(\g)=\{\alpha\in\alt^j \g^* : d\alpha=0\}$:
\begin{equation}\label{betti} 
\beta^n_k=\dim H^k(\g)=\dim Z^k(\g)+\dim Z^{k-1}(\g)-\binom{2n+2}{k-1}.
\end{equation}

In order to compute $\dim Z^k(\g)$, note that we can decompose the ideal $\u$ as $\u=\r f_2 \oplus 
\g'$, where $f_2$ is a generator of the center of $\g$. Let us denote for simplicity $\mathfrak v 
:= \g'$. 

Recalling that $\beta^n_1=\dim(\g/[\g,\g])$, it follows that $\beta^n_1=2$. Indeed, a basis of 
$Z^1(\g)$ is given by $\{f^1,f^2\}$. Due to Poincar\'e duality, we get that 
$\beta^n_{2n+1}=\beta^n_1=2$. Clearly, $\beta^n_0=\beta^n_{2n+2}=1$.

Let us consider now $2\leq k \leq 2n$. Given a closed $k$-form $\alpha \in Z^k(\g)$, it can be 
decomposed uniquely as $\alpha=f^1\^\beta+f^2\^\gamma+\delta$, with $\beta\in\alt^{k-1}\u^*$, 
$\gamma\in\alt^{k-1} \mathfrak v^*$, $\delta\in\alt^k \mathfrak v^*$, with $d\gamma=0$ and 
$d\delta=0$. Therefore, we have a one to one correspondence 
\[Z^k(\g)\longleftrightarrow (\alt^{k-1}\u^*)\times (\alt^{k-1}\mathfrak v^*\cap Z^{k-1}(\g)) 
\times (\alt^k \mathfrak v^*\cap Z^k(\g)).\]

In order to compute $\dim(\alt^k \mathfrak v^*\cap Z^k(\g))$, we consider, for $1\leq p,q\leq n$,  
the set 
\[ \left\{ (i_1,\dots,i_p,j_1,\dots,j_q): 1\leq i_1<\dots<i_p\leq n,\, 1\leq 
j_1<\dots<j_q\leq n, \, \sum^p_{r=1}i_r=\sum^q_{s=1}j_s \right\},\] 
and let us denote by $d^n_{p,q}$ its cardinal. Setting $D^n_k := \displaystyle\sum_{p+q=k} 
d^n_{p,q}$ for $k\geq2$, it can be shown that $\dim\, (\alt^k \mathfrak v^*\cap Z^k(\g))=D_k^n$. 
Let us also define $d^n_{0,0}=1$, $d^n_{1,0}=d^n_{0,1}=0$, so that $D^n_0=1,\, D^n_1=0$. Using this 
together with \eqref{betti}, we readily obtain that 
\begin{equation}\label{beta}
\beta^n_k = D^n_{k-2}+2D^n_{k-1}+D^n_k, \qquad  k \geq 2.
\end{equation}

Even though we do not obtain an explicit formula for the Betti numbers of $M_m$, using \eqref{beta} 
we are able to obtain some properties concerning the parity of these numbers.

\medskip

It is easy to see that $d^n_{p,q}=d^n_{q,p}$ and $d^n_{p,q}=d^n_{n-p,n-q}$. We can summarize some 
properties about $d^n_{p,q}$ and $D^n_k$ in the following lemma.

\begin{lema}
\
\begin{enumerate}
 \item[$\ri$] $d^n_{1,1}=d^n_{n-1,n-1}=n$ and $d^n_{k,k}\equiv \binom{n}{k}$ 
$(\operatorname{mod} 2)$.
 \item[$\rii$] If $n$ is even then $d_{1,2}^{n}=\frac{n(n-2)}{4}$ and if $n$ 
is odd then $d_{1,2}^{n}=\left(\frac{n-1}{2}\right)^2$. 
 \item[$\riii$] $D^n_k=D^n_{2n-k}$ for any $k=0,\ldots,n$, with $D_0^n=1,\, D^n_1=0, \,D_2^n=n$.
\item[$\riv$] If $k$ is odd, then $D_k^n$ is even.
\item[$\rv$] $D_{2k}^n\equiv \binom{n}{k}$ $(\operatorname{mod} 2)$.
 \end{enumerate}
\end{lema}

\medskip

Taking this lemma into account, the proof of the next result is straightforward.

\begin{prop}
\
\begin{itemize}
 \item[$\ri$] $\beta^n_0=1$, $\beta^n_1=2$ and $\beta^n_2=n+1$. 
 \item[$\rii$] If $n$ is even, then $\beta^n_3=\frac{n(n+1)}{2}$ and if $n$ is odd then 
$\beta^n_3=\frac{(n+1)^2}{2}$.
 \item[$\riii$] If $k$ is odd, then $\beta^n_k$ is even. 
 \item[$\riv$] $\beta^n_{2k}\equiv \binom{n+1}{k}$ $(\operatorname{mod} 2)$.
\end{itemize}
 \end{prop}

\medskip
 
\begin{rem}
According to Lucas' Theorem, the parity of the binomial coefficient $\binom{r}{s}$ can be 
determined from the binary representation of $r$ and $s$. Indeed, if $r=\sum_{i=1}^m a_i2^i$ and 
$s=\sum_{i=1}^m b_i2^i$ with $a_i,b_i\in\{0,1\}$, $a_m \neq 0$, then 
\[ \binom{r}{s} \equiv \prod_{i=1}^m \binom{a_i}{b_i}\, (\operatorname{mod} 2), \]
where $\binom{a_i}{b_i}=0$ if $a_i<b_i$. In particular, $\binom{r}{s}$ is even if and only there 
exists $i$ such that $a_i=0$ and $b_i=1$. For instance, if $r=2^l-1$ then $\binom{r}{s}$ is odd for 
any $0\leq s\leq r$, and if $r=2^l$ then $\binom{r}{s}$ is even for any $1\leq s\leq r-1$. 
\end{rem}

We exhibit in Table \ref{betti-dR} the Betti numbers in low dimensions. For instance, in the 
$6$-dimensional case we obtain $\beta_0=\beta_6=1,\,\beta_1=\beta_5=2, \,\beta_2=\beta_4=3, \, 
\beta_3=4$.

\begin{table}[h]
\begin{center}
\renewcommand{\arraystretch}{1.2}
\begin{tabular}{|c|c|c|}\hline
$\;n\; $ & $\dim\g$ & Betti numbers \\ \hline
 2 & 6 & (1,\,2,\,3,\,4,\,3,\,2,\,1) \\
 3 & 8 & (1,\,2,\,4,\,8,\,10,\,8,\,4,\,2,\,1) \\
 4 & 10 & (1,\,2,\,5,\,12,\,20,\,24,\,20,\,12,\,5,\,2,\,1) \\
 5 & 12 & (1,\,2,\,6,\,18,\,37,\,56,\,64,\,56,\,37,\,18,\,6,\,2,\,1) \\
 6 & 14 & (1,\,2,\,7,\,24,\,61,\,116,\,167,\,188,\,167,\,116,\,61,\,24,\,7,\,2,\,1)\\
 \hline

\end{tabular} 
\bigskip
\caption{Betti numbers}\label{betti-dR}
\end{center}
\end{table}

\bigskip

Next, we will study the adapted cohomology of the solvmanifold $M_m=\Gamma_m\backslash G$. 
Since $G$ is completely solvable, it follows from \eqref{kill_adapted} that 
$H^k_{\theta}(M_m)$ is isomorphic to the Lie algebra adapted cohomology group $H^k_{\theta}(\g)$. We 
will work in the same way as in the de Rham cohomology case.

Let $\tilde\beta^n_k=\dim H^k_{\theta}(\g)$ be the $k^{th}$ adapted Betti number of $\g$.
These numbers satisfy an equation as \eqref{betti}. Indeed, setting 
$Z^j_{\theta}(\g)=\{\alpha\in\alt^j \g^* : d_{\theta}\alpha=0\}$ we have
\begin{equation}\label{bettitilde}
\tilde\beta^n_k=\dim H^k_{\theta}(\g)=\dim Z^k_{\theta}(\g)+\dim 
Z^{k-1}_{\theta}(\g)-\binom{2n+2}{k-1}.
\end{equation}

Note that $\dim Z^1_{\theta}(\g)=2$. Let us denote by $\tilde f^1$ and $\tilde f^2$ the generators 
of $Z^1_{\theta}(\g)$ where $\tilde f^1$ coincides with $f^1$ from above. Then we can decompose the 
ideal $\u$ as $\u=\r \tilde f_2 \oplus \mathfrak v$ for some $2n$-dimensional abelian subalgebra 
$\mathfrak v$. It is easy to see that 
\[ \tilde\beta^n_1=1, \qquad \tilde\beta^n_2=n+1. \] 

Let us consider now $2\leq k \leq 2n$. Given a $d_\theta$-closed $k$-form $\alpha \in 
Z^k_{\theta}(\g)$, it can be decomposed uniquely as $\alpha=\tilde f^1\^\beta+\tilde 
f^2\^\gamma+\delta$, with $\beta\in\alt^{k-1}\u^*$, 
$\gamma\in\alt^{k-1} \mathfrak v^*$, $\delta\in\alt^k \mathfrak v^*$, with $d_\theta\gamma=0$ and 
$d_\theta\delta=0$. Therefore, we have a one to one correspondence 
\[Z^k_{\theta}(\g)\longleftrightarrow (\alt^{k-1}\u^*)\times (\alt^{k-1}\mathfrak v^*\cap Z^{k-1}_{\theta}(\g)) 
\times (\alt^k \mathfrak v^*\cap Z^k_{\theta}(\g)).\]

In order to compute $\dim(\alt^k \mathfrak v^*\cap Z^k_{\theta}(\g))$, we consider, for $1\leq p,q\leq n$,  
the set
\[ \left\{ (i_1,\dots,i_p,j_1,\dots,j_q): 1\leq i_1<\dots<i_p\leq n,\, 1\leq 
j_1<\dots<j_q\leq n, \, \sum^p_{r=1}i_r + 1=\sum^q_{s=1}j_s \right\},\] 
and let us denote by $\tilde d^n_{p,q}$ its cardinal. Setting $\tilde D^n_k := \displaystyle\sum_{p+q=k} 
\tilde d^n_{p,q}$ for $k\geq2$, it can be shown that $\dim\, (\alt^k \mathfrak v^*\cap Z^k_{\theta}(\g))=\tilde D_k^n$. Let us also define $\tilde d^n_{0,0}=\tilde d^n_{1,0}=0$, $\tilde d^n_{0,1}=1$, so that $\tilde D^n_0=0,\, \tilde D^n_1=1$.

From \eqref{bettitilde}, we readily obtain that 
\begin{equation}\label{betatilde}
\tilde \beta^n_k = \tilde D^n_{k-2}+2\tilde D^n_{k-1}+\tilde D^n_k, \qquad  k \geq 2.
\end{equation}

It is easy to see that $\tilde d^n_{p,q}=\tilde d^n_{n-q,n-p}$. Then  we have that $\tilde 
D^n_k=\tilde D^n_{2n-k}$ and therefore 
\[ \tilde \beta^n_k=\tilde \beta^n_{2n-k}, \qquad \sum_{k=0}^{2n+2}(-1)^k \tilde  \beta^n_k=0.  \]
Using \eqref{betatilde} and these facts we can compute the adapted Betti numbers in some 
low-dimensional cases, which are shown in Table \ref{betti-adapted}. 

\begin{table}[h]
\begin{center}
\renewcommand{\arraystretch}{1.2}
\begin{tabular}{|c|c|c|}\hline
$\;n\; $ & $\dim\g$ & adapted Betti numbers \\ \hline
 2 & 6 & (0,\,1,\,3,\,4,\,3,\,1,\,0) \\
 3 & 8 & (0,\,1,\,4,\,7,\,8,\,7,\,4,\,1,\,0) \\
 4 & 10 & (0,\,1,\,5,\,12,\,19,\,22,\,19,\,12,\,5,\,1,\,0) \\
 5 & 12 & (0,\,1,\,6,\,17,\,35,\,55,\,63,\,55,\,35,\,17,\,6,\,1,\,0) \\
 6 & 14 & (0,\,1,\,7,\,24,\,59,\,112,\,165,\,188,\,165,\,112,\,59,\,24,\,7,\,1,\,0)\\
 \hline

\end{tabular} 
\bigskip
\caption{Adapted Betti numbers}\label{betti-adapted}
\end{center}
\end{table}

\

\

\end{document}